\documentclass[10pt,a4paper,leqno]{amsproc}

\usepackage{amsfonts}
\usepackage{amsthm}
\usepackage{amsmath}
\usepackage{amssymb}
\usepackage[colorlinks=true]{hyperref}
\def\dis
{\displaystyle}
\def\eps{\varepsilon}

\def\ol{\overline}
\def\C{{\mathbb C}}
\def\R{{\mathbb R}}
\def\N{{\mathbb N}}
\def\Z{{\mathbb Z}}

\def\Sch{{\mathcal S}}
\def\virgp{\raise 2pt\hbox{,}}

\def\({\left(}
\def\){\right)}
\def\le{\leqslant}
\def\ge{\geqslant}
\def\<{\left\langle}
\def\>{\right\rangle}

\DeclareMathOperator{\RE}{Re}

\def\Tend#1#2{\mathop{\longrightarrow}\limits_{#1\to #2}}
\def\weakCV#1#2{\mathop{\rightharpoonup}\limits_{#1\to #2}}

\def\d{{\partial}}
\def\si{{\sigma}}

\def\F{\mathcal F}
\def\O{\mathcal O}

\theoremstyle{plain}
\newtheorem{theorem}{Theorem}[section]
\newtheorem{definition}[theorem]{Definition}
\newtheorem{assumption}[theorem]{Assumption}
\newtheorem{lemma}[theorem]{Lemma}
\newtheorem{corollary}[theorem]{Corollary}
\newtheorem{proposition}[theorem]{Proposition}
\theoremstyle{remark}
\newtheorem{remark}[theorem]{Remark}

\numberwithin{equation}{section}

\begin{document}

\title[Wigner measures for Schr\"odinger
  equations]{On the time evolution
  of Wigner measures for Schr\"odinger equations}
\author[R. Carles]{R{\'e}mi~Carles}
\address[R. Carles]{CNRS \& Universit\'e Montpellier~2\\Math\'ematiques
\\CC~051\\Place Eug\`ene Bataillon\\34095
  Montpellier cedex 5\\ France}
\email{Remi.Carles@math.cnrs.fr}
\author[C. Fermanian]{Clotilde~Fermanian-Kammerer}
\address[C. Fermanian]{LAMA UMR CNRS 8050,
Universit\'e Paris EST\\
61, avenue du G\'en\'eral de Gaulle\\
94010 Cr\'eteil Cedex\\ France}
\email{Clotilde.Fermanian@univ-paris12.fr}
\author[N. Mauser]{Norbert~J.~Mauser}
\address[N.J.  Mauser]{Wolfgang Pauli Institute c/o Fak. f. Math. \\
Univ. Wien\\
Nordbergstr. 15 \\ A 1090 Wien\\ Austria}
\email{mauser@courant.nyu.edu}
\author[H. P. Stimming]{Hans~Peter~Stimming}
\address[H. P. Stimming]{Wolfgang Pauli Institute c/o Fak. f. Math. \\
Univ. Wien}
\email{hans.peter.stimming@univie.ac.at}
\thanks{The authors acknowledge support by the Austrian Ministry of Science 
via its grant for the Wolfgang Pauli Institute 
and by the START award project of the
Austrian Science Foundation (FWF, contract No. Y-137-TEC,
as well as the European Marie Curie Project DEASE (contract
MEST-CT-2005-021122). R.C. is partially supported by the ANR project
SCASEN. 
}
\begin{abstract}
In this survey, our aim is to emphasize the main known limitations to
the use of Wigner measures for Schr\"odinger equations. 
After a short review of successful applications of Wigner measures
to study the semi-classical limit of solutions to Schr\"odinger
equations, we list some examples where Wigner measures cannot be a
good tool to describe high frequency limits. Typically, the Wigner
measures may not capture effects which are not negligible at the
pointwise level, or the propagation of Wigner measures may be an
ill-posed problem. In the latter situation, two families of
functions may have the same Wigner measures at some initial time,
but different Wigner measures for a larger time. In the case of
systems, this difficulty can partially be avoided by considering
more refined Wigner measures such as two-scale Wigner measures;
however, we give examples of situations where this quadratic
approach fails.
\end{abstract}
\maketitle
\tableofcontents

\section{Introduction}
\label{sec:intro}

In this survey,
we briefly review some successful applications of the Wigner
measures for \emph{classical limits} of Schr\"odinger equations, and discuss in
more  detail the limitations of this tool. Although most of the
material presented is essentially known, we feel that a survey
presenting the dis-advantages of Wigner measures in a clear unified
picture is timely and useful. 
\smallbreak

Wigner measures are a very valuable tool for describing high frequency 
and homogenization limits for oscillatory PDEs,
possibly with periodic coefficients. The Wigner measure is a phase space 
measure that allows to describe weak limits of quadratic quantities 
(the \emph{observables}) of a (solution) family of functions
which only converges weakly itself. 
The basic idea goes back to E.~Wigner who used such a phase space
approach in  
quantum mechanics for semi-classical approximations in 1932 \cite{Wigner32}.
In the 90's, Wigner functions and their limiting measures aroused
the interest 
of mathematicians in the USA (e.g. \cite{Papanico1}) and,  
independently, in Europe.  
As a variant of $L^ 2$ defect measures (see \cite{Tartar,GerardMDM})
such objects  
where used as a technicality in proofs in the frame of the analysis of ergodic
properties of eigenfunctions for the Dirichlet problem (see
\cite{Schn,Zeld,CdV0,HMR87}), with more systematic
studies of \emph{semiclassical measures} by P.~G\'erard and
\'E.~Leichtnam~(see \cite{GeX,GL93}, or the survey of
N.~Burq, \cite{BurqBourbaki}). 
The term ``Wigner measure" was used first in the French work ``Sur les
mesures de  
Wigner" of P.-L.~Lions and  T.~Paul \cite{LionsPaul}. 
\smallbreak

Adaptations to the case of Schr\"odinger operators  with periodic
coefficients  
and applications of the method to general problems were given by 
P.~G\'erard, P.~Markowich, N.J.~Mauser and the late F.~Poupaud
\cite{GeX}, \cite{MMP}, notably their joint paper \cite{GMMP}, 
where a general theory of the use of Wigner measures for homogenization 
limits of energy densities for several wide classes of dispersive linear PDEs 
is laid out. 
In the context of our work we also want to mention the more recent use
of Wigner  
measure for proving resolvent estimates (\cite{B1,J1,J2,FR}) following
an idea of proof by contradiction of 
\cite{Lebeau}.
\smallbreak

The method of Wigner measures allows to treat some (weakly) nonlinear
equations,  
for example the semiclassical limit of the coupled
Schr\"odinger--Poisson system,  
first done in '93 in \cite{LionsPaul} and \cite{MM93}, both works
using ``smoothed 
Wigner functions'' as a technical step to prove the non-negativity of
the limiting 
measure and both crucially depending on the use of mixed states --- an
assumption 
that could only be lifted in 1D so far \cite{ZZM}. The case of the
inclusion of 
the additional difficulty of a periodic crystal potential was solved in 
\cite{BMP2}, where the general theory of Wigner series, as
Wigner measures in the context of a Bloch decomposition of $L^2$, was laid out.

In general however, Wigner measure methods are not suitable for
treating nonlinear  
problems and even their use for linear problems has severe limitations. 
The main aim of this paper is to list some (rather explicit) examples where 
the use of Wigner measures is not appropriate and thus unveil in a clear and
concise way the inherent strengths and shortcomings of the Wigner measure 
approach for the semiclassical limit of time-dependent Schr\"odinger equations.

\subsection{Setting of the problem}

The semi-classical limit of the Schr\"odinger equation can be seen as a model
problem for the kind of homogenization limits studied by this method.
The rigorous mathematical development of
Wigner measures was motivated by this problem. By semi-classical limit we mean
the limit of the (scaled)  Planck constant
tending to $0$ in the Schr\"odinger equation, which reads
\begin{gather}\label{schr}
i\eps \d_t \psi^\eps = -\frac{\eps^2}{2}\Delta \psi^\eps  +
V(x)\psi^\eps \, \qquad
x \in \R^ d, t \in \R\\
\psi^\eps_{\mid t=0} =
\psi^\eps_I, \qquad  t \in \R.
\end{gather}
Here $\eps$ stands for the scaled Planck constant,
$\psi^\eps = \psi^ \eps(t,x)\in \C$ is the wave
function, and $V(x) \in \R$ is a given potential.
The wave function may be vector-valued and the potential then is an
 Hermitian matrix: in this situation, \eqref{schr} is a system.
Clearly, this limit
is a high frequency limit and can only exist in some weak sense.\\
From the viewpoint of physics, the main interest is not on $\psi^\eps$
itself, but in quantities which are quadratic expressions in $\psi^\eps$,
{\it e.~g.} the \emph{position density}
\begin{equation}
 \label{ndef}
n^\eps(t,x) = |\psi^\eps(t,x)|^2.
\end{equation}
Of course, this quadratic operation does not commute with the weak
$\eps \to 0$ limit of $\psi^\eps$.
The importance of getting limits
of quadratic quantities like \eqref{ndef} and the prevalence
of the distinct scale $\eps$ of oscillations in the solutions of
\eqref{schr} make Wigner measures the ``good'' tool for this problem
and generally for the homogenization of energy densities of time
dependent PDEs with a scale of oscillations.
\smallbreak

The Wigner transform of a function $f\in L^2\(\R^d\)$ is defined by
 \begin{equation} \label{wig-trafo1}
  w^\eps[f](x,\xi )
= (2\pi )^{-d}\int_{\R^d} f\(x-\eps \frac{\eta}{2}\)
\overline  f\(x+\eps \frac{\eta}{2}\)\,
e^{i \eta \cdot \xi} d\eta.
\end{equation}
In this definition, $\eps$ is an arbitrarily introduced parameter, the
scale of the Wigner transform.
$w^\eps[f](x,\xi)$ is real-valued, but in general not positive.
Now let $w^\eps(t,x,\xi)$ be the Wigner transform of $\psi^\eps(t,x)$,
then for $n^\eps(t,x)$ defined by \eqref{ndef} it holds that
\begin{equation}
  \label{densit}
n^\eps(t,x)= \int_{\R_\xi^d} w^\eps(t,x,\xi)d\xi,  \quad x \in \R^d
\end{equation}
Other quadratic quantities in $\psi^\eps$ can be obtained by
taking (higher) moments in the $\xi$-variable of $w^\eps(t,x,\xi)$.
In the case of vector-valued $f$, the Wigner transform is
matrix-valued: the product of $f$ and $\overline f$ is replaced by a
tensor product in \eqref{wig-trafo1}. One still has \eqref{densit}
provided one takes the trace of $w^\eps$.
\smallbreak

In order to study the convergence properties of $w^\eps$, in \cite{LionsPaul}
the following space was introduced:
\begin{equation*}
{\mathcal A}=\left\{ \varphi\in C_0(\R^d_x\times \R^d_{\xi})\left|(
{\mathcal F}_{\xi}\varphi)(x,\eta)\in L^1\(\R^d_{\eta};\;C_0(\R^d_x)\)\right.
\right\},
\end{equation*}
\begin{equation*}
\text{with }\
  \|{\mathcal F}_{\xi}\varphi\|_{L^1_{\eta}(C_x)}=\int_{\R^m_{\eta}}\sup_x|
{\mathcal F}_{\xi}\varphi|(x,\eta)d\eta,
\end{equation*}
where the Fourier transform is defined as follows:
\begin{equation}\label{eq:Fourier}
  \F f(\xi)=\widehat f(\xi) =
  \frac{1}{(2\pi)^{n/2}}\int_{\R^d}e^{-ix\cdot \xi}f(x)dx.
\end{equation}
It is easy to verify that ${\mathcal A}$ is an algebra of test functions and a
separable Banach algebra containing $\Sch(\R^d_x\times \R^d_{\xi})$.
This space immediately allows for a uniform estimate on the Wigner function :
\begin{proposition}
  Let $\psi^{\eps}$ be a  sequence uniformly bounded
in $L^2(\R^d)$. Then the sequence of Wigner transforms
$w^{\eps}[\psi^{\eps}]$
is uniformly bounded in ${\mathcal A}^{\prime}$.
\end{proposition}
It follows that, after selection of a subsequence,
\begin{equation}
  \label{wigmeas}
w^\eps[\psi^\eps](x,\xi) \weakCV \eps 0
w^0(x,\xi)
\quad \hbox{in }  \mathcal{A}^{'}.
\end{equation}
It can be shown that $w^0(x,\xi)$ is a non-negative measure on the
 phase space:
 the semi-classical or Wigner measure of the sequence $\psi^\eps$,
which is not necessarily unique.
Note that the Wigner transform $w^\eps$ is in general real, but may also
have negative values, whereas the limit $w^0$ is non-negative,
thus justifying the term ``Wigner measure''. In the vector-valued
 situation, one uses matrices of test functions and $w^0$ is a
 non-negative hermitian matrix of  measures, which means that
 $w^0=(w^0_{i,j})$ with $w^0_{ii}$ non-negative measure and
 $w^0_{ij}$ absolutely continuous with respect to $w^0_{ii}$ and
 $w^0_{jj}$.

We could work also with uniform bounds and convergence in
the distribution space ${\mathcal S}^{'}$ as the dual of the Schwartz 
space, as it was done e.g. in \cite{GMMP} without recourse to the
space $\mathcal{A}^{'}$.

In the scalar case, when applying  the Wigner transform to equation
\eqref{schr}, a kinetic
transport equation will result, the so-called ``Wigner equation''.
The formal limit under $\eps \to 0$ of this transport equation
leads to  the following Vlasov equation for the Wigner measure $w^0$:
\begin{equation}
 \label{lintransp}
 \d_t w^0 + \xi\cdot\nabla_x w^0
 - \nabla_x V(x)\cdot\nabla_\xi w^0  = 0
\end{equation}
If this limit can be made rigorous, and the above equation is well
posed as an initial value problem, the Wigner measure of
$\psi^\eps(t,x)$ at a positive time can be obtained by solving the
equation with the initial Wigner measure as data: the Wigner
measure is constant along the Hamiltonian trajectories
$\left(x(t,y,\eta),\xi(t,y,\eta)\right)$
\begin{equation*}
\left\{
  \begin{aligned}
   &\d_t x(t,y,\eta) =\xi \left(t,y,\eta\right) &&;\quad x(0,y,\eta)=y,\\
   &\d_t \xi(t,y,\eta) =-\nabla_x V\left(x(t,y,\eta)\right) &&;\quad
   \xi(0,y,\eta)=\eta.
  \end{aligned}
\right.
\end{equation*}
 \smallbreak

For describing the limit in the position density $n^\eps$, the following
definition will be necessary:
\begin{definition}
 A sequence $\{ f^\eps \}_\eps$ uniformly bounded
in $L^2$, is called  $\eps$-oscillatory if,
for every continuous and compactly supported function $\phi$ on $\R^d$,
\begin{equation} \label{epsosc}
\overline{\lim_{\eps \to 0}}\int_{|\xi| \ge R/\eps}
\left| \widehat{\phi f^\eps} (\xi) \right| ^2 d\xi \Tend R \infty 0.
\end{equation}
\end{definition}
This definition is presented here in the form that is used in
\cite{GMMP}, following \cite{GeX} and~\cite{GL93}; in
\cite{LionsPaul} the equivalent condition  that $1 / \eps^d
\left\lvert \widehat f^\eps
  (\xi/\eps)\right\rvert^2$
 is a
  relatively compact
sequence in ${\mathcal M} (\R^d)$ is used instead.
Heuristically this means that the wavelength of oscillations of $f^\eps$ is
at least $\eps$. A sufficient condition for \eqref{epsosc} is
$$
\exists \kappa > 0 \quad \mbox{such that} \quad \eps^\kappa D^\kappa f^\eps
\ \mbox{is uniformly bounded in} \ L^2_{\rm loc}.
$$
If $\psi^\eps$ is $\eps$-oscillatory, it is possible to pass to
the limit in \eqref{densit}, and we find
\begin{equation}
  \label{denslim}
\int_{\xi\in\R^d}w^0(x,d\xi )
\: = \: (\mbox{w--})
\lim_{\eps \to 0} n^\eps(x) =n^0(x) \quad \text{in } {\mathcal D}'.
\end{equation}
If now the Wigner measure at a positive time $t>0$ can be obtained in an
appropriate sense by the evolution equation \eqref{lintransp}, this identity
serves to get the desired limit in the density.
In the context of quantum mechanics
the strength of the method lies in the fact that  the  transport
equation of the Wigner measure \eqref{lintransp} is identical
to the transport equation from classical statistical physics.
So according to the "correspondence principle" the quantum problem 
converges to a classical problem in the limit where the "quantum 
parameter" $\eps$ vanishes.

The following diagram gives a  schematic sketch of
the method used to obtain homogenization
limits of quadratic quantities by Wigner measures.\\
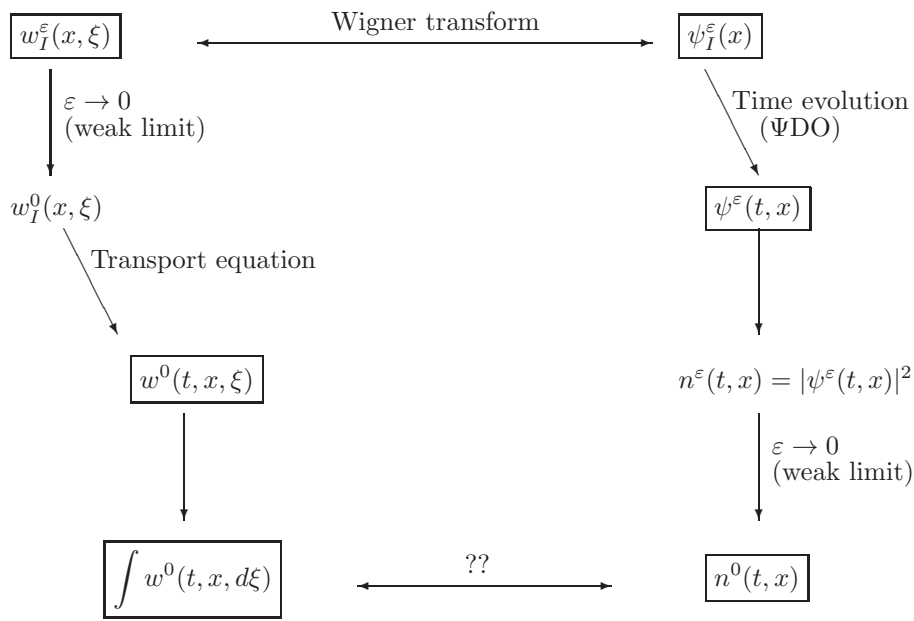
\bigskip

\begin{picture}(280,220)
\put(0,215){$\boxed{w_I^\eps(x,\xi)}$}
\put(15,205){\vector(0,-1){40}}
\put(20,190){$\eps \to 0$}
\put(20,180){(weak limit)}
\put(0,150){$w_I^0(x,\xi)$}
\put(20,145){\vector(1,-2){20}}
\put(30,130){{\rm Transport equation}}
\put(45,85){$\boxed{w^0(t,x,\xi)}$}
\put(65,75){\vector(0,-1){40}}
\put(35,10){$\boxed{\int w^0(t,x,d \xi)}$}
\put(120,220){{\rm Wigner transform}}
\put(230,215){\vector(-1,0){160}}
\put(230,215){\vector(1,0){10}}
\put(250,215){$\boxed{\psi^\eps_I(x)}$}
\put(260,205){\vector(1,-2){20}}
\put(270,190){{\rm Time evolution}}
\put(280,180){{\rm  ($\Psi$DO)}}
\put(260,150){$\boxed{\psi^\eps(t,x)}$}
\put(280,145){\vector(0,-1){40}}
\put(250,85){$n^\eps(t,x)=\lvert \psi^\eps(t,x)\rvert^2$}
\put(280,75){\vector(0,-1){40}}
\put(285,60){$\eps \to 0$}
\put(285,50){(weak limit)}
\put(260,10){$\boxed{n^0(t,x)}$}
\put(150,10){\vector(1,0){75}}
\put(150,10){\vector(-1,0){20}}
\put(170,15){??}
\end{picture}

\medskip

If it is possible to do the operations on the left hand side of
this diagram, the limit of $n^\eps$ can be found by applying
\eqref{densit} to the Wigner measure at a positive time $t>0$. So
it is necessary
\begin{itemize}\item  to have a unique Wigner
measure of the data, \item  that the transport equation for $w^0$
is a well-posed problem such that a solution exists up to a
relevant time $t>0$, \item  the $\xi$-integral (zero-th moment) of
$w^0$ at  $t>0$ must be equal to the limit $n^0(t,x)$.
\end{itemize}

For the third step to be valid, $\psi^\eps(t,x)$ has to be
$\eps$-oscillatory as stated above, which means that there must be
no oscillations at faster scales than $\eps$. If the
$\eps$-oscillatory property is imposed on the data, it will be
preserved by \eqref{schr} for positive times. The second step,
existence and uniqueness of a global solution to the transport
equation of the Wigner measure, is known to hold for a large class
of linear scalar problems, but it turns out to be much more
complicated for systems (see Section~\ref{sec:croise}); however in
nonlinear settings this point is a mostly open question.\\

Note that in the nonlinear case the non-uniqueness of the Wigner measure 
of the sequence of solutions $\psi^\eps(t,x)$ somewhat corresponds to the
non-uniqueness of the (weak) solution $w^0(t,x)$ of the limiting
nonlinear Vlasov equation; clearly, we are not able to pick one particular
solution of the ill-posed PDE by a semiclassical limit of the unique
solution of the nonlinear Schr\"odinger equation.

\subsection{Some successful applications of Wigner measures: scalar case}

For $\eps$-independent
potentials, as in \eqref{schr}, the convergence of the Wigner function to
a Wigner measure as a solution of a Vlasov equation was systematically
described in \cite{LionsPaul}. If the potential $V$ is sufficiently
smooth ($C^{1,1}$), then the Wigner measure for $(\psi^\eps)_{\eps}$
is given in terms of the Hamiltonian flow associated to $\lvert
\xi\rvert^2/2 +V(x)$. On the other hand, if $V$ is not sufficiently
smooth, uniqueness for the Hamiltonian flow fails, and the above
mentioned convergence is not guaranteed; see \cite{LionsPaul} for an
exhaustive list of examples and results.
\smallbreak

Note that actually the first mathematical study of ``semiclassical measures" 
for linear Schr\"odinger equations was done by P.~G\'erard \cite{GeX} in the
context of the setting in a crystal, where Bloch waves are considered.
\begin{equation}\label{eq:Bloch}
  i\eps\d_t \psi^\eps =-\frac{\eps^2}{2}\Delta \psi^\eps +
  V_\Gamma\(\frac{x}{\eps}\) \psi^\eps\quad ;\quad
  \psi^\eps(0,x)=\psi_0\(\frac{x}{\eps}\),
\end{equation}
where $V_\Gamma$ is lattice-periodic. In the independent later work \cite{MMP}
the names ``Wigner Bloch functions'' and  ``Wigner series'' for their limits
were coined, phase space objects that are essentially obtained via the
definition 
\eqref{wig-trafo1} by replacing the Fourier integral by a Fourier sum,
keeping the 
position variable $x$ in whole space and restricting the kinetic variable
to the torus (the Brillouin zone) as the dual of the lattice.

\smallbreak

Wigner measures have also proven successful in the study limits of
quadratic quantities of the Schr\"odinger--Poisson system:
\begin{equation}\label{nlsstate}
i\eps \d_t \psi^\eps = -\frac{\eps^2}{2}\Delta \psi^\eps  +
V^\eps_{NL}\psi^\eps\quad ;\quad
\psi^\eps_{\mid t=0} =
\psi^\eps_I,
\end{equation}
for a potential $V^\eps_{NL} = V^\eps_{NL}(|\psi^\eps|^2)$
which depends on $n^\eps=|\psi^\eps|^2$.
The main problem in applying the Wigner measure method to this equation
is the low regularity of  $w^0$, which generally is only in
${\mathcal A}^{\prime}$. This means that the transport equation for $w^0$
can only be fulfilled in a rather weak sense, so in general a nonlinear
expression in $w^0$ will have no meaning.

The applications of Wigner measures to nonlinear problems
which are available so far hold for Hartree type nonlinearities,
where the nonlinear potential $V^\eps_{NL}$ is given by
\begin{equation}
\label{hartreegen}
V^\eps_{NL} = \int U(x-y) n^\eps(y) dy
\end{equation}
for some suitable $U$ and $n^\eps$ given by \eqref{ndef}. For the
Schr\"odinger--Poisson system, one usually considers $U(x)=1/\lvert
x\rvert$ when $d=3$ (and in the absence of background ions).
The semiclassical limit results for nonlinear cases in
\cite{LionsPaul} and \cite{MM93} are possible only for the case of a
so-called ``mixed state'', i.e. by considering infinitely many
Schr\"odinger equations. In this case
convergence can take place in a stronger sense
since a uniform $L^2$-bound on the initial Wigner transform can be
imposed by choosing very particular initial states (see \cite{Mauser02} for
a comprehensive discussion).

Note however that even in the case of such strong assumptions that
is possible only if we replace the one pure state Schr\"odinger
equation by infinitely many mixed state Schr\"odinger equations, 
the results are not strong enough for obtaining a limit
that satisfies the conditions for ensuring unique classical solutions
of the Vlasov--Poisson system (see \cite{Zhidkov04}).
The theory of global weak/strong solutions of the Vlasov equations is
laid out, for example, in \cite{DiPernaLions88,LionsPerthame91} 
and \cite{Castella98,GJP00}), with famous non-uniqueness
results notably for measure valued initial data in \cite{ZhengMajda}.

 For the Schr\"odinger--Poisson system
\begin{equation}\label{eq:SP}
\left\{
  \begin{aligned}
      i\eps \d_t \psi^\eps &= -\frac{\eps^2}{2}\Delta \psi^\eps +
      V^\eps_{NL}\psi^\eps ,\\
-\Delta V^\eps_{NL}& = \left\lvert \psi^\eps\right\rvert^2-b,
  \end{aligned}
\right.
\end{equation}
with $b=b(x)\in L^1(\R^d)$, the Wigner measure associated to
$\psi^\eps$ was computed in \cite{Zhang02} for WKB-type initial data
\begin{equation*}
  \psi^\eps(0,x)= \sqrt{\rho_0^\eps(x)}e^{i\phi_0(x)/\eps}.
\end{equation*}
It was proven that the Wigner measure is given in terms of the
solution to an Euler--Poisson system, before the solution to the
latter develops a singularity. We will see in \S\ref{sec:scalar} that
in this special case, more can be said on the semi-classical limit of
$\psi^\eps$. In particular, its pointwise behavior can be described,
showing phenomena that the Wigner measures do not capture.
\smallbreak

In the special case of only one space dimension, a \emph{global in time}
description of the Wigner measure  without any mixed
state setting and the corresponding strong assumptions on the initial
data is given \cite{ZZM}. Consider
\begin{equation}
  \label{eq:ZZM}
  \left\{
    \begin{aligned}
     i\eps\d_t\psi^\eps & = -\frac{\eps^2}2\d_x^2\psi^{\eps}
+V^\eps_{NL}\psi^\eps,
 \qquad &&  x\in\R, \quad t>0,\\
 -\d_x^2 V^\eps_{NL}&=b^\eps(x)-\left\lvert \psi^\eps\right\rvert^2,  \\
 \psi^\eps|_{t=0} & =\psi^\eps_I, \qquad  && x\in\R,
    \end{aligned}
\right.
\end{equation}
in the case where $b^\eps\ge 0$ is the mollification of a function
$b\in L^1\cap L^2$, and $\psi^\eps_I$ is the mollification (with the
same mollifier) of a sequence $\varphi_\eps$ bounded in
$L^2(\R)$. Then the Wigner transform of $\psi^\eps$ converges to a
weak solution of the Vlasov--Poisson system
\begin{equation}\label{WsolVP}
\left\{
\begin{array}{l}
 \partial_tf+\xi\d_xf-E\,\d_{\xi}f=0, \quad f|_{t=0}=f_I.\\
\partial_xE=b(x)-\int_\R f\,d\xi, \quad x\in\R, \quad t\geq 0.
\end{array}
\right.
 \end{equation}

Following \cite{MMZ} we denote by {\it weak solution} on the interval $[0,T]$
any pair $(E,f)$ consisting of $E\in (BV\cap L^\infty)([0,T]\times\R)$ and 
$f\in L^\infty(\R^+, {\mathcal M}^+(\R^2))$
such that
\begin{enumerate}
\item $\forall\phi\in {\mathcal C}_c^\infty([0,T]\times\R^2),\,\exists q_\phi\in BV([0,T]\times \R)$,$$\int_\R\phi(t,x,\xi)f(t,x,\xi)\,d\xi=\d_xg_{\phi}.$$
\item $E(t,x)=\int_{-\infty}^x(b(y) -\int_{\R}f(t,y,\xi)\,d\xi)dy$ a. e.,
\item $\forall \phi\in C_c^{\infty}((0,T)\times\R^2)$,
$$\int_0^T\int_{\R^2} (\phi_tf+ \phi_x\xi f) dx\,dt -\int_0^T\!\!\int_{\R}
\tilde E\int_{\R}\phi_\xi f(\,d\xi)\,dx\,dt =0,$$
where $\tilde E(t,x)$  is  Vol'pert's symmetric average:
$$
\tilde E(t,x)
= \left\{\begin{array}{lll}
           &  E(t,x), \quad \;\;
           & \mbox{if E is approximately continuous at} \  (t,x),    \\
           &  \frac12(E_l(t,x)+E_r(t,x)) \quad
           & \mbox{if E has a jump at} \ (t,x) , \end{array}
            \right. 
$$
where $E_l(t,x),\;E_r(t,x)$ denote the left
and right limits of $E$ at $(t,x)$.
\item $\exists s>0,\; f\in C^{0,1}([0,T), H^{-s}_{loc}(\R^2)$ and 
$f(0,x,\xi)=f_I(x,\xi)$ in $H^{-s}_{loc}(\R^2)$.
\end{enumerate}
We refer to \cite{ZZM} for precise statement of the weak 
convergence result.  
We point out that the existence of weak solutions is proved in
\cite{ZhengMajda},  
but the difficulty is that there is no uniqueness of these solutions: 
a counterexample is given in \cite{MMZ}. 
We close the section by pointing out that this latter result
applies in particular to initial data which are $\eps$-independent or  of the
WKB form 
$\varphi^\eps(x)=\rho(x)e^{iS(x)/\eps}$, with $\rho\in L^2(\R)$.

\subsection{Some successful applications of Wigner measures: the case
  of systems}

In the nonlinear case, several results are available in the case of
the Schr\"odinger--Poisson system
\begin{equation}
  \label{mixschr}
  \left\{
    \begin{aligned}
     i\eps\d_t\psi_j^\eps & = -\frac{\eps^2}2\Delta\psi_j^{\eps}
+V^\eps_{NL}\psi_j^\eps,
 \qquad && j\in {\mathbb N}, \quad x\in\R^d, \quad t>0,\\
 V^\eps_{NL}&=U * \rho^\eps,  \\
 \psi_j^\eps|_{t=0} & =\phi_j^\eps, \qquad  && j\in {\N},
 \quad x\in\R^d,
    \end{aligned}
\right.
\end{equation}
where the total position density is defined as
\begin{equation}
 \label{densmixed}
\rho^\eps=\sum_{j=1}^{\infty} \lambda^\eps_j |\psi^\eps_j|^2
\end{equation}
with $\lambda^\eps_j \in \R_+$, $\sum_{j=1}^{\infty} \lambda^\eps_j =1$.
 The Wigner transform for mixed states is defined,
in analogy to \eqref{wig-trafo1}, as
\begin{equation}
  \label{Wdefmix}
w^\eps(t,x,\xi)=\frac{1}{(2\pi)^d}\int_{\R^d}e^{-i\xi \cdot y}
z^\eps\(t,x+\frac{\eps y}2,x-\frac{\eps y}2\)dy,
\end{equation}
where $z^\eps(t,r,s)$ is the so-called \emph{mixed state density matrix}
defined by
\begin{equation}
  z^\eps(t,r,s)
=\sum_{j=1}^\infty\lambda_j^\eps\psi_j^\eps(t,r)\overline{\psi_j^\eps(t,s)},
\qquad r, s\in\R^d. \label{1.5}
\end{equation}
Note that $z^\eps(t,r,s)$ is the integral kernel of the
density operator $\widehat{\rho^\eps}$ in $L^2$, which is trace class with
$\operatorname{Tr}(\widehat{\rho^\eps}) =
\sum_{j=1}^\infty\lambda_j^\eps = 1$.
For more details we refer to \cite{LionsPaul,MM93}.
A crucial property of the mixed state is the fact
that a uniform $L^2$-bound on $w^\eps_I(x,\xi)$ holds if
\begin{equation}
  \label{l2bound}
\frac{1}{\eps^3}\sum _{j=1}^{\infty} (\lambda^\eps_j)^2 \le C.
\end{equation}
This makes it possible to improve the sense of convergence for
$w^\eps(t,x,\xi)$ to weak-$L^2$.
\smallbreak

In  \cite{BMP2}, the case of the Schr\"odinger--Poisson system was
considered, with an extra  Bloch potential (as in
\eqref{eq:Bloch}). Note however that the analysis does not allow
band crossings. Indeed, the analysis in terms of Wigner measures
of systems face several difficulties in presence of eigenvalue
crossings.

\smallbreak

The analysis of \cite{GMMP} for systems cover the case of matrix
valued potentials with eigenvalues of constant multiplicity
in~(\ref{schr}). Let us denote by $\lambda_j$ the eigenvalues of
$V$ and by $\Pi_j$ the associated projectors, $1\le j\le N$.
Then any Wigner measure $w^0$ of $\psi^\eps$ decomposes as
$\displaystyle{w^0=\sum_{j=1}^N w^{0,j}}$ where the measures
$w^{0,j}$ satisfy $\Pi_j w^{0,j}\Pi_j=w^{0,j}$ and transport
equations in the distribution sense
$$\partial_t w^{0,j}+\xi\cdot\nabla_x w^{0,j} -
\nabla_x\lambda_j\cdot\nabla_\xi w^{0,j}=\left[F_j(x,\xi),w^{0,j}\right]$$
(the matrix $F_j$ depends on $\Pi_j$, $\lambda_j$, $1\le j\le
N$, see \cite{GMMP,GMMPbis} for precise formula). As
soon as the eigenvalues are not of constant multiplicities,  this
analysis is no longer valid and there may happen energy transfers
between the modes which cannot be calculated in terms of Wigner
measures. Such a phenomenon has been well-known since the works of
Landau and Zener in the 30's (see \cite{La} and \cite{Ze}). It has
been discussed from Wigner measures point of view in the articles
\cite{FG2} and \cite{Fe2} for a larger class of systems.
Eigenvalue crossings appear for Schr\"odinger systems in the frame
of quantum chemistry where one can find a large variety of
potentials presenting such features (see \cite{DYK}). In
Section~\ref{sec:croise}, for
$$\displaystyle{V(x)=\left(\begin{array}{cc}x_1 & x_2 \\ x_2 &
-x_1\end{array}\right)},$$
we explain how one can modify Wigner
measure so that a quadratic approach still is possible and give
examples where this approach fails.

\smallbreak

This article is organized as follows: we first describe in
Section~\ref{sec:scalar} the limitation of Wigner measures in
scalar situations for $\eps$ depending  potential or nonlinear
situations; then in Section~\ref{sec:croise}, we focus on the case
of systems.

\section{Limitations of Wigner measures in the scalar case}
\label{sec:scalar}

The Wigner measure method is rather limited when
nonlinear problems are treated, as can already be seen from the discussion
of the above results. Note however that some limitations are present
even in the linear case. We list below four families of problems for
which the loss of information due to the Wigner measure analysis is
rather serious.
\subsection{Ill-posedness in the linear case: wave packets}
\label{sec:Nier}

We first discuss a situation where this approach fails in a linear
setting, by recalling a case where the
Cauchy problem \eqref{lintransp} is ill-posed.
Such an example for scalar Schr\"odinger equation  is given by
F.~Nier \cite{NierX}. In that case, the potential is
$\eps$-dependent, as in \eqref{eq:Bloch}, but is decaying instead of
lattice-periodic.  The problem is that the Wigner measure
does not contain enough information on the properties of
concentration  of wave packets. Thus, Nier introduces in
\cite{NierENS} a larger phase space, and a refined Wigner transform
which takes into account the spread of the wave packets around the
point of concentration involved. Similar ideas can be found in
\cite{Miller96} and \cite{FK00}. The quadratic approach
does not fail once it is refined. The same kind of difficulties
appear in systems with matrix-valued potentials presenting
eigenvalues crossings (see Section~\ref{sec:croise}).
\smallbreak

As a particular case of
  \cite{NierX,NierENS}, consider the problem:
\begin{align}
i\eps\d_t \psi^\eps &=-\frac{\eps^2}{2}\Delta \psi^\eps
+U\left( \frac{x}{\eps}\right)\psi^\eps \label{eq:nier} ,\\
\psi^\eps(0,x) &= \frac{1}{\eps^{n/2}}u_0\left(
\frac{x}{\eps}\right), \label{eq:nierCI}
\end{align}
where $U$ is a short range potential and $u_0\in \operatorname{Ran
 }W_- = \operatorname{Ran
 }W_+$, the wave operators to the classical Hamiltonian
$-\frac{1}{2}\Delta +U$: $u_0 = W_-
 \F^{-1}\psi_-= W_+
 \F^{-1}\psi_+$. Introducing the solution $u$ to
\begin{equation*}
 i\d_t u =-\frac{1}{2}\Delta u
+U\left( x\right)u\quad ;\quad u_{\mid t=0} = u_0 ,
\end{equation*}
we see that:
\begin{equation*}
  \psi^\eps (t,x) = \frac{1}{\eps^{n/2}}u\(\frac{t}{\eps},\frac{x}{\eps}\).
\end{equation*}
By assumption, we have $u(t,x)\sim
e^{i\frac{t}{2}\Delta}\F^{-1}(\psi_\pm)$ as $t\to \pm \infty$,
which implies
\begin{equation*}
  u(t,x)\sim \frac{1}{(i t)^{n/2}}
  \psi_\pm\(\frac{x}{t}\)e^{i\frac{|x|^2}{2t}} \quad \text{as }t\to
  \pm \infty,
\end{equation*}
where we recall that the Fourier transform is normalized like in
\eqref{eq:Fourier}.
Back to the initial unknown function $\psi^\eps$, this yields:
\begin{proposition}[from \cite{NierX,NierENS}]\label{prop:Nier}
  Let $\psi^\eps$ be the solution to \eqref{eq:nier}--\eqref{eq:nierCI},
 where $U$ is a
 short range potential and $u_0\in \operatorname{Ran
 }W_- = \operatorname{Ran
 }W_+$ with $u_0 = W_- \F^{-1}\psi_-= W_+
 \F^{-1}\psi_+$. Then the Wigner measure of the (pure) family
 $\psi^\eps$ is given by
\begin{equation*}
  w^0(t,x,\xi)=\left\{
      \begin{aligned}
\frac{1}{|t|^n}\left| \psi_- \(\frac{x}{t}\)\right|^2 dx
\otimes
\delta_{\xi =\frac{x}{t}} \quad &\text{if }t<0\, ,\\
\frac{1}{|t|^n}\left|\psi_+ \(\frac{x}{t}\)\right|^2 dx
\otimes \delta_{\xi =\frac{x}{t}} \quad &\text{if }t>0\, .
\end{aligned}
\right.
\end{equation*}
\end{proposition}
As mentioned in \cite{NierENS}, given two functions $\psi_-$ and
$\psi_-'$ such that $|\psi_-|\equiv| \psi_-'|$, one should not
expect $|\psi_+|\equiv |\psi_+'|$, even in space dimension one.
This is a first hint that the propagation of Wigner measures is an
ill-posed problem in this context. This argument is made more
precise in \cite{NierX} (this example was not resumed in the
complete paper \cite{NierENS}). In space dimension one, assume
that the potential $U$ is even $U(x)=U(-x)$, and fix $T>0$. Let
$\psi^\eps$ be a solution to \eqref{eq:nier} such that its Wigner
measure satisfies
\begin{align*}
&w^0(-T,x,\xi) = \delta_{x=-x_0}\otimes \delta_{\xi=\xi_0} \, ,\\
&w^0(T,x,\xi) = (1-R^2)\delta_{x=x_0}\otimes \delta_{\xi=\xi_0} +
R^2 \delta_{x=-x_0}\otimes \delta_{\xi=-\xi_0}\, ,
\end{align*}
where $R=|R(\xi_0)|=|R(-\xi_0)|$ is the reflection coefficient
related to the scattering operator $S=W_+^\ast W_-$. Define
$\breve\psi^\eps (t,x)=\psi^\eps(-t,-x)$. Then since $U$ is even,
$\breve\psi^\eps$ solves the same equation as $\psi^\eps$, and its
Wigner measure satisfies
\begin{align*}
&\breve w^0(-T,x,\xi) = (1-R^2)\delta_{x=-x_0}\otimes
\delta_{\xi=\xi_0} + R^2
\delta_{x=x_0}\otimes \delta_{\xi=-\xi_0} ,\\
&\breve w^0(T,x,\xi) = \delta_{x=x_0}\otimes \delta_{\xi=\xi_0}.
\end{align*}
Define $\widetilde \psi^\eps = \psi_1^\eps +\psi_2^\eps$ where the
$\psi^\eps_j$'s solve \eqref{eq:nier} and whose initial Wigner
measures are such that:
\begin{align*}
w^0_1(-T,x,\xi) = (1-R^2)\delta_{x=-x_0}\otimes \delta_{\xi=\xi_0}
\quad ;\quad w^0_2(-T,x,\xi) =  R^2 \delta_{x=x_0}\otimes
\delta_{\xi=-\xi_0}  .
\end{align*}
Then $\breve w^0(-T,x,\xi) =\widetilde w^0(-T,x,\xi) $, and unless
$R=0$ or $1$,  $\breve w^0(T,x,\xi) \not =\widetilde
w^0(T,x,\xi)$:
\begin{equation*}
  \widetilde w^0(T,x,\xi) = \( (1-R^2)^2 +R^4\)\delta_{x=x_0}\otimes
  \delta_{\xi=\xi_0} +2R^2(1-R^2)\delta_{x=-x_0}\otimes
  \delta_{\xi=-\xi_0} .
\end{equation*}
Therefore, the Cauchy problem for the propagation of Wigner measures
is ill-posed in this case: knowing the Wigner measure at time $t=-T$
does not suffice to determine it at time $t=+T$.

\subsection{Caustic crossing and ill-posedness in a nonlinear case}
\label{sec:illposed2}

We now give another example on a nonlinear problem, taken from
\cite{CaIUMJ}, which may be
viewed as a nonlinear counterpart of the above example.
Consider a Schr\"odinger equation with
power-like nonlinearity:
\begin{equation}\label{eq:IUMJfnl}
  i\eps\d_t \psi^\eps = -\frac{\eps^2}{2}\Delta \psi^\eps + \eps^2
  |\psi^\eps|^{4/d} \psi^\eps \quad ;\quad
\psi^\eps\big|_{t=0}= a_0(x)e^{-i\frac{|x|^2}{2\eps}}\,,
\end{equation}
where $x\in \R^d$, $d\ge 1$. The asymptotic behavior of $\psi^\eps$ is
given in \cite{CaIUMJ} for any time. Note that other powers in the
nonlinearity are also considered, provided that the power of $\eps$ in
front of the nonlinearity is well chosen.
\begin{proposition}[\cite{CaIUMJ}]\label{prop:IUMJfnl}
  Let $a_0\in\Sch(\R^d)$, and consider $\psi^\eps$ solution to
  \eqref{eq:IUMJfnl}. Then the asymptotic behavior of $\psi^\eps$ is
  given by $\|\psi^\eps(t)-v^\eps(t)\|_{L^2(\R^d)}\to 0$ as $\eps \to 0$,
  where:
  \begin{equation*}
    v^\eps(t,x)=\left\{
      \begin{aligned}
        \frac{1}{(1-t)^{d/2}}
a_0\left(\frac{x}{1-t} \right)e^{i\frac{|x|^2}{2\eps(t-1)}}& \ \textrm{
  if }t<1,\\
\frac{e^{-id\frac{\pi}{2}}}{(t-1)^{d/2}}
}{\mathcal Z}a_0\left(\frac{x}{t-1} \right)e^{i\frac{|x|^2}{2\eps(t-1)} & \
  \textrm{ if }t>1.
      \end{aligned}
\right.
  \end{equation*}
Here ${\mathcal Z} = \F \circ S\circ \F^{-1}$, where $\F$ stands for the
  Fourier transform \eqref{eq:Fourier}, and $S$
  denotes the scattering operator associated to $i\d_t u
  +\frac{1}{2}\Delta u = |u|^{4/d}u$ (see
  e.g. \cite{Caz}).
\end{proposition}
Like in Subsection~\ref{sec:weaklyNL}, we infer the
Wigner measure of $\psi^\eps$:
\begin{equation*}
  w^0(t,x,\xi)=\left\{
      \begin{aligned}
\frac{1}{|t-1|^d}\left| a_0 \(\frac{x}{1-t}\)\right|^2 dx \otimes
\delta_{\xi =\frac{x}{t-1}} \quad &\text{if }t<1,\\
\frac{1}{|t-1|^d}\left|{\mathcal Z}a_0
\(\frac{x}{1-t}\)\right|^2 dx \otimes
\delta_{\xi =\frac{x}{t-1}} \quad &\text{if }t>1.
\end{aligned}
\right.
\end{equation*}
Unless $|a_0|^2 =\left\lvert {\mathcal Z}a_0\right\rvert^2$, $w^0$ has
a jump at the caustic crossing. The Wigner measure solves the
transport equation with a singular source term:
\begin{equation*}
  \d_t w^0 +\xi \cdot \nabla_x w^0 =\delta_{x=0}\otimes
\( \left\lvert {\mathcal Z}a_0(\xi)\right\rvert^2 -|a_0(\xi)|^2\)d\xi
\otimes \delta_{t=1}.
\end{equation*}
The pathology is even more serious:
the following result was established in \cite{CaWigner} in the
one-dimensional setting, and its proof extends to any space dimension.
\begin{proposition} \label{prop:illposed2}
Let $d\ge 1$.\\
$(1)$ There exists $a_0\in \Sch(\R^d)$ such that the Wigner measure $w^0$
associated to $\psi^\eps$ solving \eqref{eq:IUMJfnl} is
discontinuous at $t=1$:
\begin{equation*}
  \lim_{t\to 1^-}w^0(t,dx,d\xi) \not =
  \lim_{t\to 1^+}w^0(t,dx,d\xi) .
\end{equation*}
$(2)$ There exist two (pure) families $(\psi^\eps_j)_{0<\eps\le 1}$, $j=1,2$,
solutions
to \eqref{eq:IUMJfnl}  (with different initial profiles $a_{0,j}$),
whose Wigner measures $w^0_j$ are such that $w^0_1=w^0_2$ for $t<1$
and $w^0_1\not =w^0_2$ for $t>1$.
\end{proposition}
\begin{proof}[Sketch of the proof]
The point is to show that
one can find $(a_{0,j})_{j=1,2}$ such that $|a_{0,1}|=|a_{0,2}|$ and
$|{\mathcal Z} (a_{0,1})|\not = |{\mathcal Z}
(a_{0,2})|$.
Since very few properties of the scattering operator
$S$ are available, the first two terms of its asymptotic expansion
near the origin are computed, following the approach of
\cite{PG96} (the first term is naturally the identity). In the case of
$L^2$-critical nonlinear Schr\"odinger equation considered here, we
have, for $\psi_-\in L^2(\R^d)$ and $0<\delta\ll 1$,
$$\displaylines{\qquad
    S \(\delta\psi_-\) = \delta\psi_- -
 i\delta^{1+4/d}\int_{-\infty}^{+
 \infty} U_0(-t)\( |U_0(t)\psi_-|^{4/d}U_0(t)\psi_-\)dt
\hfill\cr\hfill
  +
 \O_{L^2(\R^d)}\(\delta^{1+8/d}\),
  \qquad\cr}$$
where $\dis U_0(t)=e^{i\frac{t}{2}\Delta}$. The proof of the above
identity relies on Strichartz estimates and a bootstrap argument; a
complete proof can be found in \cite{COMRL}. The idea is then to
consider
\begin{equation*}
  a_{0,1}=a_0\in \Sch(\R^d),\ \ a_{0,2}=a_0 e^{ih},\text{ with }h\in
  C^\infty\(\R^d;\R\).
\end{equation*}
We proceed as in \cite{CaWigner}. Denote
\begin{equation*}
  P(\psi_-) = -i\int_{-\infty}^{+
 \infty} U_0(-t)\( |U_0(t)\psi_-|^{4/d}U_0(t)\psi_-\)dt.
\end{equation*}
Obviously,
\begin{equation*}
\left\lvert \F\circ S \(\delta\psi_-\)\right\rvert^2 = \delta^2
\left\lvert \widehat \psi_-\right\rvert^2 + 2\delta^{2+4/d} \RE \(
\overline{\widehat  \psi_-} \widehat{P\psi_-}\) +\O\(
\delta^{2+8/d}\),
\end{equation*}
and we have to prove that we can find $\psi_-\in \Sch(\R^d)$, and
$h\in C^\infty(\R^d;\R)$, such that
\begin{equation*}
  \RE \( \overline{\F  \psi_-} \F\(P\psi_-\)\)\not = \RE \(
  \overline{\F\(\psi_h\)} \F\(P\(\psi_h\)\)\)=:R(\psi_-,h),
\end{equation*}
where $\psi_h$ is defined by
\begin{equation*}
  \widehat \psi_h(\xi) = e^{ih(\xi)}\widehat \psi_-(\xi).
\end{equation*}
If this was not true, then for every $\psi_-\in \Sch(\R^d)$, the
differential of the map $h\mapsto R(\psi_-,h)$
would be zero at every smooth, real-valued function $h$. An elementary
but tedious computation shows that
\begin{equation*}
  D_h R(\psi_-,0)(h)\not\equiv 0,
\end{equation*}
with $h(x)=|x|^2/2$ and $\psi_-(x)=e^{-|x|^2/2}$. The computations
uses the fact that the evolution of Gaussian functions under the
action of the free
Schr\"odinger group can be computed explicitly. We refer to
\cite{CaBook} for more detailed computations.
\end{proof}
\begin{remark}
  A similar result can be established in the case of the Hartree
  equation with harmonic potential studied in \cite[Sect.~5]{CMSSIAP}. For
  $\gamma>1$ and $a_0\in \Sch(\R^n)$, consider:
\begin{equation}
\label{eq:r3fnl}
i\eps \d_t \psi^\eps +\frac{\eps^2}{2}\Delta \psi^\eps =
 \frac{|x|^2}{2}\psi^\eps +
 \eps^{\gamma} \(|x|^{- \gamma}\ast |\psi^\eps|^2\)\psi^\eps\quad ; \quad
 \psi^\eps_{\mid t=0} = a_0 .
\end{equation}
Like in \cite{CaIHP} in the case of a power-like nonlinearity, the
harmonic potential causes focusing at the origin periodically in
time; outside the foci, the nonlinearity is negligible, and near
the foci, the harmonic potential is negligible while the
nonlinearity is not. Like for Proposition~\ref{prop:IUMJfnl}, its
influence is described in average by the scattering operator
associated to $i \d_t u +\frac{1}{2}\Delta u =
 \(|x|^{- \gamma}\ast |u|^2\)u$,
whose existence was proven in \cite{GV80,HT87a}. Following the same
approach as in \cite{CaWigner}, one can prove the analogue of
Proposition~\ref{prop:illposed2} in the case of Eq.~\eqref{eq:r3fnl}.
\end{remark}

\subsection{WKB analysis and weak perturbations}
\label{sec:weaklyNL}
Consider a Schr\"odinger equation with an $\O(\eps)$
perturbation and a WKB data:
\begin{equation*}
  i\eps \d_t \psi^\eps = -\frac{\eps^2}{2}\Delta \psi^\eps
  +V(t,x)\psi^\eps +\eps
  F^\eps(t,x) \psi^\eps \quad ;\quad
\psi^\eps\big|_{t=0}= a_0(x)e^{i\phi_0(x)/\eps},
\end{equation*}
where $V$ and $F^\eps$ are real-valued.
If we assume that $F^\eps$ has an expansion as $\eps\to 0$ of the form
$F^\eps = F_0 +\eps F_1+\ldots$, then a formal WKB analysis yields
$\psi^\eps\approx a e^{i\phi/\eps}$, where:
\begin{equation*}
  \left\{
\begin{aligned}
  \d_t \phi +\frac{1}{2}|\nabla \phi|^2 +V(t,x)=0 &\quad ; \quad
  \phi\big|_{t=0}=\phi_0 .\\
\d_t a +\nabla \phi\cdot \nabla a +\frac{1}{2}a\Delta \phi = -i F_0 a&\quad
  ; \quad a\big|_{t=0}=a_0 .
\end{aligned}
\right.
\end{equation*}
The first equation is the eikonal equation. Classically, it is solved
locally in space and time, provided that $V$ and $\phi_0$ are smooth,
by considering the Hamiltonian flow
\begin{equation}
  \label{eq:hamilton}
\left\{
  \begin{aligned}
   &\d_t x(t,y) =\xi \left(t,y\right) &&;\quad x(0,y)=y,\\
   &\d_t \xi(t,y) =-\nabla_x V\left(t,x(t,y)\right) &&;\quad
   \xi(0,y)=\nabla \phi_0(y).
  \end{aligned}
\right.
\end{equation}
See e.g. \cite{DG,GrigisSjostrand}. In addition, if $\phi_0$ and $V$
are subquadratic, then one can find a local existence time which is
uniform with respect to $x\in\R^d$:
\begin{assumption}\label{hyp:geom}
  We assume that the potential and the initial phase are smooth and
  subquadratic:
  \begin{itemize}
  \item $V\in C^\infty(\R_t\times \R^d_x)$, and $\partial_x^\alpha V\in
  L^\infty_{\rm loc}(\R_t ;L^\infty(\R^d_x))$ as soon as $|\alpha|\ge 2$.
  \item $\phi_0\in C^\infty( \R^d)$, and $\partial_x^\alpha \phi_0\in
  L^\infty(\R^d)$ as soon as $|\alpha|\ge 2$.
  \end{itemize}
\end{assumption}
\begin{lemma}[from \cite{CaBKW}]\label{lem:hj}
  Under Assumption~\ref{hyp:geom}, there exist $T>0$ and a unique
  solution $\phi_{\rm eik}\in C^\infty([0,T]\times\R^d)$ to:
  \begin{equation}
    \label{eq:eik}
    \partial_t \phi_{\rm eik} +\frac{1}{2}|\nabla \phi_{\rm eik}|^2
    +V(t,x)=0\quad ;\quad
    \phi_{{\rm eik} \mid t=0}=\phi_0 .
  \end{equation}
This solution is subquadratic: $\partial_x^\alpha \phi_{\rm eik} \in
  L^\infty([0,T]\times\R^d)$ as soon as $|\alpha|\ge 2$.
\end{lemma}
Essentially, for $0\le t\le T$, the map $y\mapsto x(t,y)$ given by
\eqref{eq:hamilton} is a diffeomorphism of $\R^d$: for $0\le t\le T$,
the Jacobi determinant
\begin{equation*}
  J_t(y) ={\rm det}\nabla_y x(t,y)
\end{equation*}
is bounded away from zero.
For $0\le t\le T$,
no caustic is formed yet, and the second
equation is a transport equation. It is an ordinary differential
equation along the rays of geometric optics: introduce $A$ given by
\begin{equation*}
  A(t,y) := a \left(t, x(t,y)
  \right)\sqrt{J_t(y)}.
\end{equation*}
The transport equation is then equivalent, for $0\le t\le T$, to:
\begin{equation}\label{eq:transportEDO}
  \d_t A(t,y) = -i F_0\(t, x(t,y)\)  A(t,y); \quad A(0,y)=a_0(y) .
\end{equation}
We note that since $F_0$ is real-valued, the modulus of $A$ is
independent of $t\in [0,T]$. The influence
of $F_0$ shows up through a phase shift, whose wavelength is $\O(1)$:
\begin{equation*}
  A(t,y) = a_0(y) \exp\(-i \int_0^t F_0\(\tau,x(\tau,y)\)d\tau\).
\end{equation*}
Back to $a$, we find
\begin{equation*}
  a(t,x) = \frac{1}{\sqrt{J_t\(y(t,x)\)}}
a\( y(t,x)\)\exp\(-i \int_0^t F_0\(\tau,x(\tau,y(t,x))\)d\tau\) ,
\end{equation*}
where $y(t,x)$ stands for the inverse mapping of $y\mapsto x(t,y)$.
Wigner measures ignore this integral, since
it corresponds to a phase whose wavelength is large compared to
$\eps$.

The above approach is very general, and includes linear problems. For
instance, we may take $F^\eps(x,t)= f(x)\in \Sch(\R^d)$. We could also
consider perturbations of order $\O(\eps^\alpha)$ with $0<\alpha\le 1$
instead of $\O(\eps)$, and follow the same line of reasoning. We illustrate
this general statement with examples, corresponding to weakly
nonlinear phenomena. The term
``weakly'' means that the nonlinearity does not appear in the eikonal
equation (the geometry of propagation is the same as in the linear
case), but is present in the leading order transport equation.

Consider the nonlinear Schr\"odinger equation in $\R^d$
\begin{equation}\label{eq:wnlgo}
  i\eps\d_t \psi^\eps = -\frac{\eps^2}{2}\Delta \psi^\eps +
  V(t,x)\psi^\eps+\eps
  f\(|\psi^\eps|^{2}\) \psi^\eps \quad ;\quad
\psi^\eps\big|_{t=0}= a_0 e^{i\phi_0/\eps},
\end{equation}
where $f\in C^\infty(\R_+;\R)$ and Assumption~\ref{hyp:geom} is
satisfied. Here,
\begin{equation*}
  F^\eps = f\(|\psi^\eps|^{2}\)
\end{equation*}
is a nonlinear function of $\psi^\eps$. However, we see from
\eqref{eq:transportEDO} that, at leading order, the modulus of
$\psi^\eps$ is independent of time, so that \eqref{eq:transportEDO}
turns out to be a linear ordinary differential equation.
In \cite{CaBKW}, the asymptotic behavior of $\psi^\eps$ is given for
$t\in [0,T]$, by:
\begin{proposition}[from \cite{CaBKW}]\label{prop:sub}
  Let
  $f\in C^\infty(\R_+;\R)$, $a_0\in \Sch(\R^d)$, and let
  Assumption~\ref{hyp:geom} be satisfied.  Then for all
  $\eps\in ]0,1]$,
  \eqref{eq:wnlgo} has a unique solution $\psi^\varepsilon \in
  C^\infty([0,T]\times \R^d)\cap C([0,T];H^s)$ for all $s>d/2$ ($T$ is
  given by
  Lemma~\ref{lem:hj}). Moreover,
  \begin{equation*}
    \left\| \psi^\varepsilon -  a e^{iG}e^{i\phi_{\rm eik}
        /\varepsilon}\right\|_{L^\infty([0,T]; L^2\cap L^\infty) } \to 0\quad
    \text{as }\varepsilon \to 0.
  \end{equation*}
The functions $a$ and $G$ are given by
\begin{equation*}
  \begin{aligned}
    a(t,x) &= \frac{1}{\sqrt{J_t\(y(t,x)\)}}a_0\left(y(t,x)\right),\\
  G(t,x) &= -\int_0^t
  f\left(J_s(y(t,x))^{-1}\left|a_0(y(t,x)) \right|^2\right)ds.
  \end{aligned}
\end{equation*}
\end{proposition}
In particular, the unique Wigner measure for $\psi^\eps$ is given by:
\begin{equation*}
  w^0(t,x,\xi) = \frac{1}{J_t\(y(t,x)\)} \left\lvert
  a_0\left(y(t,x)\right) \right\rvert^2dx\otimes \delta_{\xi = \nabla
  \phi_{\rm eik} (t,x)}.
\end{equation*}
It is independent of the nonlinearity $f$, and therefore, does not
take the nonlinear effect, causing the non-trivial presence of $G$,
into account.
\smallbreak

A similar analysis could be carried out by replacing the above local
nonlinearity by a non-local Hartree type term. We shall simply exhibit
an explicit example in such a framework. In \cite{CMSSIAP}, the
following Hartree
equation with a harmonic potential was considered:
\begin{equation}\label{eq:CMSSIAPognl}
i\eps \d_t  \psi^\eps =-\frac{\eps^2}{2}\Delta \psi^\eps +
\frac{|x|^2}{2}\psi^\eps +
\eps \(|x|^{-\gamma}\ast |\psi^\eps|^2\)\psi^\eps \quad ; \quad
\psi^\eps_{\mid t=0} = a_0 ,
\end{equation}
with $0<\gamma<1$ and $x\in\R^d$ for $d\ge 2$. The solution to
\eqref{eq:eik} is explicit in this case:
\begin{equation*}
  \phi_{\rm eik}(t,x) = -\frac{\lvert x\rvert^2}{2}\tan t,\quad
  t\not\in \frac{\pi}{2}+\pi\Z.
\end{equation*}
\begin{proposition}[{\cite[Prop.~4.1]{CMSSIAP}}]\label{prop:nlwkb}
Let $d\ge 2$, $a_0\in\Sch(\R^d)$, and $0<\gamma <1$. Let $\psi^\eps$
be the solution to \eqref{eq:CMSSIAPognl}. Define (for any $t$)
\begin{equation*}
g(t,x) =- \(|x|^{-\gamma}\ast |a_0|^2\)(x) \int_0^t \frac{d\tau}{|\cos
\tau|^\gamma} .
\end{equation*}
$\bullet$ For $0\le t <\pi/2$, the following asymptotic relation holds:
\begin{equation*}
\sup_{0\le \tau\le t} \left\|\psi^\eps(\tau ,x)- \frac{1}{(\cos
\tau)^{n/2}} a_0\(\frac{x}{\cos\tau}\)
e^{-i\frac{|x|^2}{2\eps}\tan \tau +ig\(\tau, \frac{x}{\cos\tau}\)}
\right\|_{L^2_x} \Tend \eps 0 0  .
\end{equation*}
$\bullet$ For $\pi/2< t \le\pi$,
\begin{equation*}
\sup_{t\le \tau\le \pi} \left\|\psi^\eps(\tau ,x)-
\frac{e^{-in\frac{\pi}{2}}}{(\cos \tau)^{n/2}}
a_0\(\frac{x}{\cos\tau}\) e^{-i\frac{|x|^2}{2\eps}\tan \tau
+ig\(\tau \frac{x}{\cos\tau}\)} \right\|_{L^2_x} \Tend \eps 0 0 .
\end{equation*}
For any time $t\in [0,\pi]\setminus\{\frac{\pi}{2}\}$, the Wigner
measure $w^0$ associated to the
(pure) family $(\psi^\eps)_{0<\eps\le 1}$ is given by:
\begin{equation*}
  w^0(t,x,\xi) = \frac{1}{|\cos t|^n}\left| a_0\left( \frac{x}{\cos t}\right)
\right|^2 dx\otimes \delta_{\xi =-x \tan t} .
\end{equation*}
\end{proposition}
As mentioned in \cite{CMSSIAP}, the asymptotic description could be
pursued to any time. Like in the first example, the
Wigner measure is the same as in the linear case and ignores leading
order nonlinear effects measured by $g$.
\smallbreak

  The phenomenon we described in this paragraph is also present in the
  main result of \cite{CMSJSP}, where a weakly nonlinear perturbation
  of \eqref{eq:Bloch} is considered:
  \begin{equation*}
   i\eps \d_t \psi^\eps = -\frac{\eps^2}{2}\Delta \psi^\eps
  +V(t,x)\psi^\eps + V_{\Gamma}\(\frac{x}{\eps}\)\psi^\eps + \eps
  \left\lvert \psi^\eps\right\rvert^{2\si}\psi^\eps,
  \end{equation*}
where $V$ is as above, $V_\Gamma$ is lattice-periodic, and $\si\in
\N$. Before the formation of caustics, nonlinear effects show up at
leading order through a self-modulation (phase shift), which may be
viewed in this case as a nonlinear Berry phase.
\smallbreak

We emphasize the fact that the self-modulation described in this
paragraph is not bound to the Schr\"odinger equation. In \cite{CR}, a
nonlinear wave equation is considered in a weakly nonlinear
r\'egime. If in \cite{CR}, we consider  a purely imaginary coupling
constant, that is, a nonlinear wave equation of the form
\begin{equation*}
  \d_t^2 u - \Delta u + i\lvert \d_t u\rvert^{p-1}\d_t u=0,
\end{equation*}
then we meet the same phenomenon as in this paragraph if, following
the notations of \cite{CR}, $P_-=0$ or $P_+=0$ (this
corresponds to polarized initial data).

\subsection{Supercritical WKB r\'egime}
\label{sec:super}

To conclude on the case of scalar equations, consider the case
\begin{equation}
  \label{eq:NLSsurcritique}
  i\d_t \psi^\eps = -\frac{\eps^2}{2}\Delta \psi^\eps + \left\lvert
  \psi^\eps\right\rvert^2 \psi^\eps \quad ;\quad
  \psi^\eps(0,x)=a_0^\eps(x)e^{i\phi_0(x)/\eps}.
\end{equation}
We consider the case of a cubic, defocusing nonlinearity for
simplicity; the approach recalled below can be extended to a wider
class a nonlinearities, from \cite{AC-BKW,ThomannAnalytic}. Assume
that the initial amplitude $a_0^\eps$ has an asymptotic expansion of
the form
\begin{equation*}
  a_0^\eps = a_0+\eps a_1 +\O\(\eps^2\), \text{ as }\eps\to 0,
\end{equation*}
where the functions $a_0$ and $a_1$ are independent of $\eps$. The
case considered in \S\ref{sec:weaklyNL} was critical as far as WKB
analysis is concerned: nonlinear effects are present in the transport
equation (which determines the leading order amplitude), and it would
not be the case if the power $\eps$ in front of the nonlinear term was
replaced by $\eps^{\kappa}$, with $\kappa>1$ (see \cite{CaBKW}). The
above equation is supercritical as far as
WKB analysis is concerned. If we seek
\begin{equation*}
  \psi^\eps(t,x)= \({\tt a}_0(t,x) +\eps {\tt a}_1(t,x)+
  \O\(\eps^2\)\)e^{i\phi(t,x)/\eps},
\end{equation*}
then plugging this expression into \eqref{eq:NLSsurcritique} and
ordering the powers of $\eps$ yields:
\begin{align*}
 \O\left(\varepsilon^0\right):&\quad \partial_t \phi
 +\frac{1}{2}|\nabla\phi|^2 + |{\tt a}_0|^2=0,\\
\O\left(\varepsilon^1\right):&\quad \partial_t {\tt a}_0 +\nabla\phi
 \cdot \nabla {\tt a}_0 +\frac{1}{2}{\tt a}_0\Delta \phi =
2i{\RE}\left({\tt a}_0\overline{{\tt
 a}_1}\right){\tt a}_0.
\end{align*}
We see that there is a strong
coupling between the phase and the main amplitude: ${\tt a}_0$ is present
in the equation for $\phi$. Moreover, the above system is not closed:
$\phi$ is determined in function of ${\tt a}_0$, and ${\tt a}_0$ is
determined in function of ${\tt a}_1$. Even if we pursued the cascade
of equations, this phenomenon would remain: no matter how many terms
are computed, the system is never closed (see \cite{PGX93}). This is a
typical feature of supercritical cases in nonlinear geometrical optics
(see \cite{CheverryBullSMF,CG05}).
\smallbreak

Suppose however that we know $\phi$, and that the rays associated to
$\phi$ do not form an envelope for $t\in [0,T]$ (we prefer not to
speak of caustic in that case; see \cite{CaJHDE}). Then along these
rays, and following the same approach as in \ref{sec:weaklyNL}, we
see that the equation for ${\tt a}_0$ is of the form
\begin{equation*}
  D_t {\tt a}_0 = 2i{\RE}\left({\tt a}_0\overline{{\tt
 a}_1}\right){\tt a}_0.
\end{equation*}
Therefore, the modulus of ${\tt a}_0$ is constant along rays, and the
coupling between ${\tt a}_0$ and ${\tt a}_1$ is present only through a
phase modulation for ${\tt a}_0$. As in \ref{sec:weaklyNL}, this phase
modulation is not trivial in general: a perturbation of the initial
data at order $\O(\eps)$ in \eqref{eq:NLSsurcritique} leads to a
modification of $\psi^\eps$ at leading order $\O(1)$ for positive
times. This phenomenon was called \emph{ghost effect} in a slightly
different context \cite{Sone}. This discussion is made rigorous in
\cite{CaBKW}, after rewriting the original idea of E.~Grenier
\cite{Grenier98}. Note also that in the above discussion, we have seen
that we do not need to know ${\tt a}_1$ to determine $|{\tt a}_0|^2$:
this is strongly related to the following observation. Set
$(\rho,v)=(|{\tt a_0}|^2,\nabla \phi)$. Then the above system for
$\phi$, ${\tt a}_0$ and ${\tt a}_1$ implies
\begin{equation*}
  \left\{\
    \begin{aligned}
      &\d_t v + v\cdot \nabla v + \nabla \rho =0\quad && ;\quad v_{\mid
      t=0}=\nabla \phi_0,\\
& \d_t \rho +\nabla\cdot\(\rho v\)=0\quad && ;\quad \rho_{\mid
      t=0}=|a_0|^2.
    \end{aligned}
\right.
\end{equation*}
This system is a polytropic compressible Euler equation. WKB analysis
can be justified so long as the solution to this system remains smooth
(\cite{Grenier98,CaBKW}), and the Wigner measure is given by
\begin{equation*}
  w^0(t,x,\xi)= \rho(t,x)dx\otimes\delta_{\xi=v(t,x)}.
\end{equation*}
See also \cite{ZhangJPDE}.
\smallbreak

The above remarks can be applied also when $a_1$ depends on
$\eps$. Formally, if we replace $a_1$ by $\eps^{-\delta}a_1$,
$0<\delta<1$, then we should at least replace ${\tt a}_1$ by
$\eps^{-\delta}{\tt a}_1$, and also reconsider the remainder
$\O(\eps^2)$. Forget this last point. Mimicking the above discussion,
the interaction between ${\tt a}_0$ and ${\tt a}_1$ leads to a phase
modulation for ${\tt a}_0$, of order $ \eps^{-\delta}$: ``rapid''
oscillations appear. However, such oscillations are not detected at
the level of Wigner measures: the $x$-component of the Wigner measure
ignores this modulation, because it is of modulus one, and the
$\xi$-component cannot see it, because its wavelength $\eps^\delta$ is
too large compared to $\eps$. However, this phenomenon is everything
but negligible at the level of the wave functions, since it causes
instabilities:
\begin{theorem}[from \cite{CaARMA}]\label{theo:ARMA}
Let $d\ge 1$, $a_0 ,a_1\in\Sch({\R}^d)$,
$\phi_0\in C^\infty({\R}^d;{\R})$, where $a_0$, $a_1$ and
$\phi_0$ are independent of $\eps$, and $\nabla \phi_0 \in
H^s({\R}^d)$ for every $s\ge 0$.  Let $u^\eps$ and
$v^\eps$ solve the initial value problems:
\begin{align*}
i\eps \d_t u^\eps + \frac{\eps^2}{2}\Delta u^\eps &= |u^\eps|^2 u^\eps
\ ; \ u^\eps\big|_{t=0}= a_0e^{i\phi_0/\eps} .\\
i\eps \d_t v^\eps + \frac{\eps^2}{2}\Delta v^\eps &=|v^\eps|^2 v^\eps
\ ; \ v^\eps\big|_{t=0}= \(a_0+\eps^ka_1\) e^{i\phi_0/\eps} ,
\end{align*}
for some $0<k<1$.
Assume that $\RE \(\overline a_0 a_1\)\not \equiv 0$. Then we can find
$0<t^\eps\Tend \eps 0 0$ such that
\begin{equation*}
  \liminf_{\eps\to 0}\left\lVert u^\eps\(t^\eps\)-
  v^\eps\(t^\eps\)\right\rVert_{L^2\cap L^\infty} >0.
\end{equation*}
More precisely, this mechanism occurs as soon as $t^\eps
\gtrsim \eps^{1-k}$.
In particular,
\begin{equation*}
\frac{\left\| u^\eps - v^\eps
  \right\|_{L^\infty([0,t^\eps];L^2)}}{\left\| u^\eps_{\mid t=0} - v^\eps_{\mid
  t=0}
  \right\|_{L^2}}\to +\infty \quad \text{as }\eps\to 0 .
\end{equation*}
\end{theorem}
As pointed out above, the Wigner measures ignore this instability.
One could argue that this instability mechanism does not affect
the quadratic quantities, and thus may not be physically relevant
(the same remark could be made about \S\ref{sec:weaklyNL}). Note
however that this phenomenon occurs for very small times, when WKB
analysis is still valid. When WKB ceases to be valid (in this
case, this corresponds to the appearance of singularities in Euler
equations), the approach must be modified. We have seen in
\S\ref{sec:illposed2} that even for a weaker nonlinearity,
nonlinear effects can alter drastically the Wigner measures when
rays of geometric optics form an envelope. The consequence shown
in \S\ref{sec:illposed2} was an ill-posedness result for the
propagation of Wigner measures. Even if Theorem~\ref{theo:ARMA}
may not seem relevant for Wigner measures, it may very well happen
that for larger times, it causes another ill-posedness phenomenon.
\smallbreak

The same discussion remains valid in the case of the
Schr\"odinger--Poisson system \eqref{eq:SP}. It was proven in
\cite{Zhang02} that so long as the solution to a corresponding
Euler--Poisson remains smooth, it yields the Wigner measure associated
to $\psi^\eps$.  However, in \cite{AC-SP}, the approach of Grenier was
adapted to the case of \eqref{eq:SP}: the above instability
mechanism can be inferred in this case as well, and the discussion
remains the same.
\smallbreak

It turns out that this instability mechanism can be met in the case of
a weaker (as far as WKB analysis is concerned) nonlinearity. Using a
semi-classical conformal transform, we infer from
Theorem~\ref{theo:ARMA}:
\begin{corollary}[from \cite{CaARMA}]\label{cor:weak}
Let $d\ge 2$, $1<\alpha<d$, and $a_0,a_1\in{\Sch}({\R}^d)$ independent
 of $\eps$. Let $u^\eps$ and
 $v^\eps$ solve the initial value problems:
\begin{align*}
i\eps \d_t u^\eps + \frac{\eps^2}{2}\Delta u^\eps  &= \eps^\alpha
|u^\eps|^2 u^\eps
\ ; \ u^\eps(0,x)= a_0(x)e^{-i\lvert x\rvert^2/(2\eps)} ,\\
i\eps \d_t v^\eps + \frac{\eps^2}{2}\Delta v^\eps &=
\eps^\alpha|v^\eps|^2v^\eps
\ ; \ v^\eps(0,x)= \(a_0(x) +\eps^{1-\alpha/d}
a_1(x)\)e^{-i\lvert x\rvert^2/(2\eps)} .
\end{align*}
Assume that $\RE (
  \overline a_0 a_1)\not \equiv 0$.
There
exist $T^\eps  \Tend \eps 0 1^{-}$  and $0<\tau^\eps \Tend \eps 0 0$ such that:
\begin{equation}\label{eq:instabweak}
\left\| u^\eps - v^\eps \right\|_{L^\infty([0,T^\eps];L^2)}\Tend \eps
0 0 \quad
;\quad
\liminf_{\eps \to 0}\left\| u^\eps - v^\eps
\right\|_{L^\infty([0,T^\eps+\tau^\eps];L^2)}>0 .
\end{equation}
\end{corollary}
Consider the case $d=2$, and compare with the result of
\S\ref{sec:illposed2}. The assumption $\alpha<2$ implies that
supercritical nonlinear effects occur near the focusing time $t=1$
(see also \cite{BZ} for a similar result),
but not before since $\alpha>1$. This case, along with
Proposition~\ref{prop:illposed2}, illustrates the above discussion:
Wigner measures do not capture the above instability mechanism, but
for larger times, this instability may affect the Wigner measures.

 \section{Limitation of the Wigner measures in the case of systems}
\label{sec:croise}

We consider now systems of Schr\"odinger equations coupled by a
matrix-valued potentials. Such systems appear in the framework of
 Born-Oppenheimer approximation, where
 the dynamics of molecules can approximately be
reduced to matrix-valued Schr\"odinger equations on the nucleonic
configuration space. We consider
\begin{equation}\label{eq:Schro}
\left\{
\begin{aligned}
i\eps\partial_t \psi^\eps&=-\frac{\eps^2}{2}\Delta\psi^\eps
+V(x)\psi^\eps,\;\;(t,x)\in\R^+\times\R ^d,\;\;
\psi^\eps\in\C^j,\\
\psi^\eps_{|t=0}&=\psi^\eps_0\in
L^2(\R^d,\C^j),
\end{aligned}\right.
\end{equation}
where the small parameter $\eps$ is the square root of the
ratio of the electronic mass on the average mass of molecule's
nuclear. Of typical interest is the situation where
\begin{equation*}
d=2,\;\;j=2,\;\;V(x)=\
\begin{pmatrix}
x_1 & x_2 \cr x_2 & -x_1\cr
\end{pmatrix},
\end{equation*}
 which is referred to as a codimension $2$ crossing (see
 \cite{Hag94}). Indeed, the
 eigenvalues of $V$ are $\pm |x|$,
 and they cross on a codimension $2$ subspace $S$ of $\R^d$,
 $$S=\{x_1=x_2=0\}.$$
 The main features of eigenvalue crossings appear on this example
 which is simple enough so that precise computations can be
 performed.

 According to the analysis of \cite{GMMP} and taking advantage of the
  fact that the eigenvalues are of multiplicity $1$ outside the
  crossing set
  $S$, any  Wigner measure
  $w^0$ of the family $\psi^\eps(t)$ is a $2$ by $2$ matrix of
  measures which splits into two parts
 $$w^0(t,x,\xi) =w^{0,+}(t,x,\xi)\Pi^+(x)+w^{0,-}(t,x,\xi)\Pi^-(x),$$
 where $\Pi^+(x)$ and $\Pi^-(x)$ are the spectral projectors associated
 respectively with the eigenvalues
  $+ |x|,\,-|x|$ of $V(x)$ and $w^{0,+},\,w^{0,-}$ are scalar
  non-negative
 measures. Moreover,
outside the crossing set, $w^{0,\pm}$ propagate along the classical
 trajectories of $\frac{|\xi|^2}{2}\pm|x|$ according to
$$\partial_t w^{0,\pm}+\xi\cdot\nabla_x w^{0,\pm}\,\mp\,\frac{x}{
  |x|}\cdot\nabla_\xi w^{0,\pm}=0.$$

Let us focus on initial data which have only one Wigner
measure $w^0_I$ and which are microlocally localized on two points
of  the phase space, so that
\begin{equation}\label{eq:mu0}
w^0_I(x,\xi)=\sum_{j\in\{+,-\}}^{}\,a_0^j\,\delta(x-x_0^j)\otimes\delta(\xi-\xi_0^j)\,\Pi^j(x).
\end{equation}
By \cite{GMMP}, as long as the classical trajectories
$\left(x^\pm(t),\xi^\pm(t)\right)$, that is the curves such that
\begin{equation*}
\left\{
  \begin{aligned}
   &\d_t x^\pm(t) =\xi^\pm \left(t\right) &&;\quad x^\pm(0)=x_0^\pm,\\
   &\d_t \xi^\pm(t) =\mp \frac{x^\pm(t)}{| x^\pm(t)|}&&;\quad
   \xi^\pm(0)=\xi_0^\pm.
  \end{aligned}
\right.
\end{equation*}
do not reach the crossing set $S$,  $\psi^\eps(t)$ has only one
Wigner measure $w^0(t,x,\xi)$ which is of the form
$$
w^0(t)=a^+(t)\,\delta(x-x^+(t))\otimes\delta(\xi-\xi^+(t))\,\Pi^+
+a^-(t)\,\delta(x-x^-(t))\otimes\delta(\xi-\xi^-(t))\,\Pi^-
$$
with $a^\pm(t)=a_0^\pm$ as long as the trajectories have not
reached the crossing. Such situations are precisely studied in
\cite{FL}: it is proved in Proposition~1 of \cite{FL} that under
the assumptions
\begin{align*}
  &x^\pm_0\wedge\xi^\pm_0=0\quad ;\quad |\xi_0^-|^2>2|x_0^-| \quad
  ;\quad \xi_0^-\cdot
x_0^-<0,\\
& \xi_0^+\cdot\frac{ x_0^+}{
|x_0^+|}+\sqrt{|\xi_0^+|^2+2|x_0^+|}=-\xi_0^-\cdot
\frac{x_0^-}{|x_0^-|}-\sqrt{|\xi_0^-|^2-2|x_0^-|}:=t^*,
\end{align*}
the
curves $\left(x^\pm(t),\xi^\pm(t)\right)$ reach $S$ at the same
time $t^*$ with a non-zero speed $\xi^\pm(t^*)$ and one can choose
$x_0^\pm$, $\xi_0^\pm$ so that they reach $S$ at the same point
$(0,\xi^*)$ with $\xi^*\not=0$. One then has
$$x^+(t^*)=x^-(t^*)=0,\;\;\xi^+(t^*)=\xi^-(t^*)=\xi^*.$$
The trajectories are given on $[0,t^*]$ by
\begin{align*}
  x^+(t)&=-\frac{t^2}{2}\frac{x_0^+}{|x_0^+|}+t\xi_0^++x_0^+ &&;\quad
\xi^+(t)=-t\frac{x_0^+}{|x_0^+|}+\xi_0^+,\\
x^-(t)&=\frac{t^2}{2}\frac{x_0^-}{|x_0^-|}+t\xi_0^-+x_0^- &&;\quad
\xi^-(t)=t\,\frac{x_0^-}{|x_0^-|}+\xi_0^-.\end{align*} Then, there
exists $t_1>t^*$ such that during $(t^*,t_1)$, both trajectories
do not meet $S$ again. We will suppose during all this section
that we are in this situation and for simplicity, we assume
$|x_0^\pm|=1$. We  are concerned on the Wigner measure of
$\psi^\eps(t)$ for $t\in(t^*,t_1)$ which cannot be described by
the transport equations above mentioned. \smallbreak

 The work \cite{FG1} proves that the sole knowledge of the initial
 Wigner measure is not
 enough to determine the Wigner measure of the solution after the
 crossing time $t^*$.
 A second level of observation is required, and the decisive fact is
 the way the data concentrates
 on the classical trajectories entering in $S$ with respect to the
 scale $\sqrt\eps$. Since all
 these trajectories are included in the set
 $$J=\{x\wedge \xi:=x_1\xi_2-x_2\xi_1=0\},$$
  it is enough to study the concentration of $\psi^\eps (t)$ on $J$.

 We come up with the first aspect of our purpose, which is that the
 sole microlocalization
 consisting in working in the phase space is not enough. The complete
 analysis of $\psi^\eps(t)$
 after a crossing point requires a second microlocalization: we add to
 the phase space variables
 $(x,\xi)$ a new variable $$\eta\in\overline\R:=\R\cup\{+\infty,-\infty\}$$
 which describes the spread of the wave packet on both sides of $J$
 with respect to the scale
 $\sqrt\eps$. Then, one defines two-scale Wigner measures  whose
 projections on the phase
  space are the Wigner measure. These measures have been first
 introduced in  \cite{Miller96}
 and developed in \cite{FG1}; the reader may also find in \cite{Fe4} a
 survey on the topic.
\smallbreak

 In the following, we explain the resolution of our problem by means
 of two-scale Wigner
 measures (according to \cite{FG1}). This explains why initial data may
 have the same Wigner measure
 and   generate solutions of the Schr\"odinger equation which have
 different Wigner measures
 after a crossing time. We give examples of this fact. Then, in a
 second section, we shall
 explain and illustrate the limits of the two-scale Wigner measures
 approach: some situations
 cannot be studied
 with the sole knowledge of the two-scale Wigner measures.

 \subsection{When one needs a second microlocalization}

One defines a two-scale Wigner measure $w^{0,(2)}(t,x,\xi,\eta)$
of $\psi^\eps(t)$ associated with $J$, as a weak limit in
${\mathcal D}'$ of the two-scale Wigner functional defined on
$\R^2_x\times\left(\R^2_\xi\setminus\{0\}\right)\times\ol\R_\eta$
by
\begin{equation*}
  w^{\eps,(2)}\left[\psi^\eps(t)\right](x,\xi,\eta)
=w^{\eps}\left[\psi^\eps(t)\right](x,\xi)\otimes\delta\left(\eta-\frac{x\wedge\xi}{\sqrt\eps}\right)
\end{equation*}
that we test against smooth functions $a(x,\xi,\eta)$ which are
 compactly supported in $(x,\xi)$, uniformly
 with respect to $\eta$, and coincide for $\eta$ big enough with an
 homogeneous function of degree $0$  in $\eta$, denoted by
 $a_\infty(x,\xi,\eta)$.
The knowledge of $w^{0,(2)}(t,x,\xi,\eta)$ determines the Wigner
measures of $\psi^\eps
 (t)$ above $J$ according to
 $$\forall a\in
C_0^\infty\left(\R^2_x\times\left(\R^2_\xi\setminus
\{0\}\right)\right),\;\;\<w^0\,{\bf
 1}_ J,a\>=
 \int_{\R^2_x\times\left(\R^2_\xi\setminus
\{0\}\right)\times\ol\R_\eta}a(x,\xi)\,d w^{0,(2)}.$$ Let us
suppose that the data $\psi_0^\eps$ has a unique two-scale Wigner
measure $w^{0,(2)}_I$ of the form
 $$w^{0,(2)}_I(x,\xi,\eta)=\sum_{j\in\{+,-\}}^{}\,\delta(x-x_0^j)
\otimes\delta(\xi-\xi_0^j)\otimes \nu_{in}^j(\eta)\,\Pi^j(x) ,$$
where $\nu_{in}^\pm(\eta)$ are scalar positive Radon measures on
$\ol\R$, and $x_0^\pm$, $\xi_0^\pm$ are as before. Define
\begin{equation}\label{def:T}
T(\eta)={e}^{-\pi\lvert \eta\rvert^2/\lvert\xi^*\rvert^3}.
\end{equation}
Then, Theorems~2 and~3 in~\cite{FG1} yield
\begin{proposition}\label{prop1}
For any $t\in[0,t^*)\cup(t^*,t_1)$, $\psi^\eps(t)$ has a unique
two-scale Wigner measure $w^{0,(2)}(t)$ which satisfies:\\
$\bullet$ For $ t\in[0,t^*)$,
\begin{eqnarray*}
  w^{0,(2)}(t,x,\xi,\eta) =\sum_{j\in\{+,-\}}^{}\, \delta(x-x^j(t))\otimes
\delta(\xi-\xi^j(t))\otimes \nu_{in}^j(\eta)\,\Pi^j(x).
\end{eqnarray*}
$\bullet$ For $t\in(t^*,t_1)$,
\begin{eqnarray*}
  w^{0,(2)}(t,x,\xi,\eta) = \sum_{j\in\{+,-\}}^{}\, \delta(x-x^j(t))\otimes
\delta(\xi-\xi^j(t))\otimes \nu_{out}^j(\eta)\,\Pi^j(x).
\end{eqnarray*}
Moreover if $\nu_{in}^+$ and $\nu_{in}^-$ are singular on
$\{|\eta|<\infty\}$, the link between the incident measures
$(\nu_{in}^+,\nu_{in}^-)$ and the outgoing ones
$(\nu_{out}^+,\nu_{out}^-)$ is given by
\begin{equation*}
  \begin{pmatrix}
    \nu_{out}^+\cr\nu_{out}^-\cr
  \end{pmatrix}
=
\begin{pmatrix}
  1-T(\eta) & T(\eta) \cr T(\eta) & 1-T(\eta)\cr
\end{pmatrix}
\begin{pmatrix}
  \nu_{in}^+\cr\nu_{in}^-\cr
\end{pmatrix}
.
\end{equation*}
\end{proposition}
Let us examine the consequence of this Proposition for initial
data of the form
\begin{align*}
  \psi^\eps_{0,\ell}&=\eps^{-\beta_\ell}\Phi \left(\frac{x-x_0^+}{
      \eps^{\beta_\ell}}\right){\rm
exp}\left(\frac{i}{ 2\eps}r_0^+|
x|^2\right)E^+(x),\;\;\ell\in\{1,2,3\},\\
& \text{with }\beta_1<\frac{1}{2}=\beta_2<\beta_3,
\end{align*}
where $\xi_0^+=r_0^+x_0^+$, $\Phi\in C_0^\infty(\R^2)$,  and $E^+$ is
a smooth bounded function  such that $\Pi^+E^+=E^+$ on the support of
      $\psi^\eps_{0,\ell}$, and
$\|E^+\|_{\C^2}=1$. These three families have the same Wigner
measure $w^0_I$ of the form~\eqref{eq:mu0} with
$$a^+_0=\|\Phi\|^2_{L^2},\;\;a_0^-=0.$$
 Simple calculus (see  \cite[Section~3.2]{FL} for details) shows that the
two-scale Wigner measures associated with $J$ for each of these
families are different. Let us denote by $w^{0,(2)}_{I,\ell}$ the
two-scale Wigner measure of $\psi^\eps_{0,\ell}$,
$\ell\in\{1,2,3\}$. If $v=(v_1,v_2)$, we denote by $v^\perp $ the
vector $v^\perp=(-v_2,v_1)$. Then, we have
\begin{align*}
  w^{0,(2)}_{I,1}(x,\xi,\eta)&=\|\Phi\|_{L^2}^2\,\delta(x-x_0^+)\otimes
\delta(\xi-\xi_0^+)\otimes\delta(\eta)\,\,\Pi^+(x),\\
w^{0,(2)}_{I,2}(x,\xi,\eta)&=\delta(x-x_0^+)\otimes\delta(\xi-\xi_0^+)\otimes
{\bf
  1}_{\eta\in\R}
 \gamma(\eta)\,d\eta\,\,\Pi^+(x), \\
w^{0,(2)}_{I,3}(x,\xi,\eta)&=\delta(x-x_0^+)\otimes\delta(\xi-\xi_0^+)\otimes\left(\gamma^-
\delta(\eta-\infty) +\gamma^-\delta(\eta+\infty)\right)\,
\Pi^+(x),
\end{align*}
with
\begin{align*}
  \xi_0^+=r_0^+ x_0^+\ ;\ \ \gamma(\eta)=  \frac{1}{
(2\pi)^2}\left(\int_\R\left\lvert\widehat\Phi(rx_0^ ++\eta
(x_0^+)^\perp)\right\rvert^2\,d r\right)\ ; \ \ \gamma^\pm=\int_{\R^\pm}
\gamma(\eta)d\eta.
\end{align*}
As a consequence of the propagation of two-scale Wigner measures
along the classical trajectories (see  \cite{FL}), we get
different behaviors after the crossing time $t^*$.

\begin{corollary}
For $\ell\in\{1,2,3\}$, the family $\psi^\eps_\ell(t)$ solution
to~\eqref{eq:Schro} with the initial data $\psi^\eps_{0,\ell}$ has
a unique Wigner measure $w^{0}_\ell(t)$ such that
\begin{eqnarray*}
w^0_\ell(t,x,\xi,\eta) & = &
a^+_\ell(t)\,\delta(x-x^+(t))\otimes\delta(\xi-\xi^+(t))\,\Pi^+(x)\\
&  &
+\;\;a^-_\ell(t)\,\delta(x-x^-(t))\otimes\delta(\xi-\xi^-(t))\,\Pi^-(x).
\end{eqnarray*}
Moreover, for all $ t\in[0,t^*)$
$$a^+_1(t)=a^+_2(t)=a^+_3(t)=\|\Phi\|^2_{L^2},\;\;a^-_1(t)=a^-_2(t)=
a^-_3(t)=0,$$
and for all $t\in(t^*,t_1)$,
\begin{align*}
&  a^+_1(t)=\|\Phi\|^2_{L^2},\;\; a^-_1(t)=0,&\\
&a^+_2(t)=\left(1-T(\eta)\right)\gamma(\eta),\;\;
a^-_2(t)= T(\eta)\gamma(\eta) ,&\\
 &a^+_3(t)=0,\;\;
a_3^-(t)=\|\Phi\|^2_{L^2}.&
\end{align*}
\end{corollary}
In the first case, the mass propagates along the classical
  trajectory associated with the  mode~$+$, while in the third case the
  mass switches
from the mode~$+$ to the mode $-$. In the second case, the mass
  parts between both modes~$+$ and~$-$.

 \subsection{When the quadratic approach fails.}

    The assumption of singularity on the incident measures is crucial
    in Proposition~\ref{prop1}.
    Indeed, if it fails,
    the two-scale Wigner measures after the crossing point cannot be
    calculated in terms of the
    incident ones.
    Our aim is now to explain and to illustrate that similar Wigner
    measures (and even similar
    two-scale Wigner measures) can generate different measures after
    the crossing point. The example given here is precisely discussed
    in  \cite{Fe07}.

We consider
\begin{align*}
\psi^\eps_0&=  \eps^{-\beta d/2}\Phi \left(\frac{x-x_0^+}{
\eps^{\beta}}\right){\exp}\left(\frac{i}{ 2\eps}r_0^+\left\lvert
x-\eps^{\alpha^+} \omega_0^+\right\rvert^2\right)E^+(x)\\
&+\eps^{-\beta d/2}\Psi \left(\frac{x-x_0^-}{
\eps^{\beta}}\right){\rm exp}\left(\frac{i}{2\eps}r_0^-\left\lvert
x-\eps^{\alpha^-} \omega_0^-\right\rvert^2\right)E^-(x),
\end{align*}
 where
$$0<\alpha^\pm\le 1/2\quad ;\quad 0<\beta<1/2\quad ;\quad
 \omega_0^+,\omega_0^-\in\R^d\quad ;\quad\xi_0^\pm=r_0^\pm\,x_0^\pm,$$
 and $\Phi$ and $\Psi$ are smooth, compactly supported, functions on $\R^d$,
$E^+(x)$ (resp. $E^-(x)$) is a smooth bounded function  such that
on the support of $\Phi\left(\frac{x-x_0^+}{ \eps^{\beta}}\right)$
(resp. of $\Psi\left(\frac{x-x_0^-}{ \eps^{\beta}}\right)$) one
has $\Pi^\pm(x) E^\pm(x)=E^\pm(x)$ and $\|E^\pm(x)\|_{\C^2}=1$. We
shall focus on both situations~$\omega_0^+\not=\omega_0^-$
and~$\omega_0^+=\omega_0^-$. We set
$$c^{+,in}=\|\Phi\|_{L^2},\;\;c^{-,in}= \|\Psi\|_{L^2} $$
and we suppose
$$\eta_0^\pm:=-r_0^\pm(x_0^\pm\wedge \omega_0^\pm)\not=0.$$

\begin{proposition}\label{prop:ex}
The family $\psi^\eps(t)$ solution to~\eqref{eq:Schro} with the
initial data $\psi^\eps_{0}$ has a unique Wigner measure
$w^{0}(t,x,\xi)$ such that
\begin{itemize}
\item For $t\in[0,t^*)$,
$$w^0(t,x,\xi)=\sum_{j\in\{+,-\}}^{}\,c^{j,in}
\delta\left(x-x^j(t)\right)\otimes\delta\left(\xi-\xi^j(t)\right)\Pi^j(x).$$
\item For $t\in(t^*,t_1)$,
$$w^0(t,x,\xi)=\sum_{j\in\{+,-\}}^{}\,c^{j,out}
\delta\left(x-x^j(t)\right)\otimes\delta\left(\xi-\xi^j(t)\right)\Pi^j(x),$$
\end{itemize}
where the coefficients $c^{j,out}$ depend on the position of
$\alpha^+$ and $\alpha^-$ with respect to $1/2$:
$$\begin{array}{|l|c|c|}
\hline & c^{+,out} & c^{-,out} \\ \hline \alpha^-,\alpha^+<1/2 &
c^{+,in} & c^{-,in} \cr  \alpha^+<\alpha^-=1/2 &
c^{+,in}+T(\eta_0^-)c^{-,in} &
c^{-,in}\left(1-T(\eta_0^-)\right)\\
\alpha^-<\alpha^+=1/2 &
c^{+,in}\left(1-T(\eta_0^+)\right) & c^{-,in}+T(\eta_0^+)c^{+,in} \\
\left\{
\begin{array}{l}\alpha^+=\alpha^-=1/2\\ \eta_0^+\not = \eta_0^-
\end{array}
 \right.   &\begin{array}{c}
c^{+,in}\left(1-T(\eta_0^+)\right)\\+T(\eta_0^-)c^{-,in}\end{array}
&\begin{array}{c}c^{-,in}\left(1-T(\eta_0^-)\right)\\
+T(\eta_0^+)c^{+,in}\end{array}\\
 \hline\end{array}.$$
 If $\alpha^+=\alpha^-=1/2$ and $\eta_0^+= \eta_0^-:=\eta_0$,
there exists $\rho_0\in\R^+$ and $\phi_0\in\R$ such that
\begin{equation*}
  \left\{
    \begin{aligned}
   c^{+,out}&=c^{+,in}\left(1-T(\eta_0)\right)+T(\eta_0)c^{-,in}+\rho_0\,{\rm
cos}(\phi_0),\\
c^{-,out}&=c^{-,in}\left(1-T(\eta_0)\right)+T(\eta_0)c^{+,in}-\rho_0\,{\rm
cos}(\phi_0).
    \end{aligned}
\right.
\end{equation*}
Besides, for any $\phi\in\R$, if one turns $\Psi$ into
${e}^{i\phi}\Psi$, then one has
\begin{equation}\label{cout}
\left\{
\begin{aligned}
c^{+,out}&=c^{+,in}\left(1-T(\eta_0)\right)+T(\eta_0)c^{-,in}+\rho_0\,{\rm
cos}(\phi_0-\phi),\\
c^{-,out}&=c^{-,in}\left(1-T(\eta_0)\right)+T(\eta_0)c^{+,in}-\rho_0\,{\rm
cos}(\phi_0-\phi).
\end{aligned}
\right.
\end{equation}
\end{proposition}
Here again, we see that the interaction between both incident modes
cannot be described only by the knowledge of the Wigner measure: one
needs to know the two-scale Wigner measure. Indeed, for each case
described in the above array, the two-scale Wigner measure is
different:
$$w^{0,(2)}(t,x,\xi,\eta)=\sum_{j\in\{+,-\}}^{}\,c^{j,out}
\delta\left(x-x^j(t)\right)\otimes\delta\left(\xi-\xi^j(t)\right)\otimes
\nu^j(\eta)\Pi^j(x)
,$$
with $\nu^j(\eta)$ supported on $\infty$ for $\alpha^j<1/2$ and
$\nu^j(\eta)$ supported on $\eta_0^j$ if $\alpha^j=1/2$.  Therefore,
the situations of the array enters in the range of validity of
Proposition~\ref{prop1}. In the last situation where
$\alpha^+=\alpha^-=1/2$ and $\eta_0^+=\eta_0^-$, the two incident wave
functions interact and some coupling term  appears. By modifying the
initial data, turning $\Psi$ into ${e}^{i\lambda}\Psi$, one does not
modify the two-scale Wigner measure but one can transform the coupling
(and even suppress it) so that the outgoing Wigner measures change.
 The same thing happens for any initial data provided it has the same
two-scale Wigner measure as our example.

For proving Proposition~\ref{prop:ex}, we use  the normal form one
can obtain
    for~\eqref{eq:Schro}, and which is
    crucial for proving
    Proposition~\ref{prop1}. The reader can refer to \cite{CdV1} for
    the more elaborate result
    for general systems
    presenting codimension~2 crossings (see also \cite{CdV2} for the
    case of hermitian
    matrix-valued symbol). One can find in \cite{FG1} a weaker result
     which covers the case of Schr\"odinger equation \eqref{eq:Schro},
    and is enough for
     calculating Wigner measures.
    Through a change of symplectic coordinates in space-time phase space,
    and a change of unknown by
    use of a Fourier Integral operator, one reduces to the system:
    \begin{equation}\label{eq:nf}
    \frac{\eps}{ i}\partial_s u^\eps=
    \begin{pmatrix}
      s & z_1 \cr z_1 &
    -s\cr
    \end{pmatrix}
u^\eps,\;\;(s,z=(z_1,z_2))\in\R^{3}.
    \end{equation}
    We outline two important features of this normal form. On the one
    hand, this normal form is microlocal, near the crossing point
    $(t^*,0,\tau^*,\xi^*)$ where $\tau^*$ is the energy variable
    $\tau^*=\frac{|\xi^*|^2}{2}$. On the
    other hand,
    space time coordinates are involved by the canonical
    transform (i.e. by the change of symplectic coordinates)
    $$\kappa:(t,x,\tau,\xi)\mapsto (s,z,\sigma,\zeta).$$
    Through $\kappa$, the crossing set $S$ becomes $S=\{s=z_1=0\}$ and the set
    $J=\{x\wedge\xi=0\}$ becomes $J=\{z_1=0\}$. Moreover, the
    classical trajectories in space time
    coordinates $\left(r,x^\pm(r),\tau,\xi^\pm(r)\right)_{r>0}$
    (where $\tau=-\frac{1}{
    2}{|\xi^\pm(r)|^2}\pm|x^\pm(r)|=\text{Const.}$) which enters
    in $S$ maps on the curves
    $\left(r,0,z_2,\pm|r|,\zeta_1,\zeta_2\right)_{r>0}$ and the
    link between $z_1$ and $x\wedge\xi$ is given by
    $z_1=e(t,x,\tau,\xi)\,
    x\wedge\xi$, where
$e$ is a smooth function such that
$e(t^*,0,\tau^*,\xi^*)=|\xi^*|^{-3/2}$. Finally, the Wigner
measure is invariant by this change of coordinates, and if we
denote by $\lambda$ the additional variable of the two-scale
Wigner measure in the variables $(s,z)$, and by $\eta$ the
corresponding variables in $(x,\xi)$, we have
$\lambda=e(t,x,\tau,\xi) \eta$ because of the link between $z_1$
and $(t,x,\tau,\xi)$.

 We focus now on system \eqref{eq:nf}. This
    system of o.d.e.  is simple enough so that direct calculations are
    possible.
  This has been done by Landau and Zener in the 30's (see \cite{La}
    and \cite{Ze}). We use
  the description of $u^\eps$ near $z_1=0$
  as stated in
  \cite[Proposition~9]{FG1}, which is obtained
   by means of stationary phase method. One can also find a
    resolution of this system in
   \cite{Hag94}, where special functions are used.

\begin{proposition}\label{prop:asympt}
 There exist families of vectors of $\C^2$,
 $\alpha^\eps=(\alpha^\eps_{1},\alpha^\eps_2)$,
 $\omega^\eps=(\omega^\eps_{1},\omega^\eps_2)$,
    such that, as $\eps$
     goes to $0$ and for $\frac{z_1}{\sqrt\eps}$ bounded, \\
$\bullet$ For $s<0$,
   $$u^\eps_1(s,z)  =
    {e}^{i\frac{s^{2}}{2\eps}} \left|\frac{s}{\sqrt \eps}\right|
    ^{i\frac{z_1^2}{2\eps}}
    \alpha^\eps_{1}+o(1),\;\;
    u^\eps_2(s,z)  =
    {e}^{-i\frac{s^{2}}{2\eps}}\left |\frac{s}{\sqrt \eps}\right|^{-i\frac{z_1^2}{2\eps}}
    \alpha^\eps_{2}+o(1).$$
$\bullet$ For $s>0$,
$$u^\eps_1(s,z)  =
    {e}^{i\frac{s^{2}}{2\eps}} \left|\frac{s}{\sqrt \eps}
\right|^{i\frac{z_1^2}{2\eps}}
    \omega^\eps_{1}+o(1),\;\;
   u^\eps_2(s,z) =
    {e}^{-i\frac{s^{2}}{2\eps}} \left|\frac{s}{\sqrt \eps}\right|
    ^{-i\frac{z_1^2}{2\eps}}
    \omega^\eps_{2}+o(1).$$
Moreover $\left(\begin{array}{c}\omega^\eps_{1}\\
    \omega^\eps_{2}\end{array}\right)=S\left(\frac{z_1}{\sqrt \eps}\right)
    \left(\begin{array}{c}\alpha^\eps_{1}\\
    \alpha^\eps_{2}\end{array}\right)$ with
    \begin {equation}\label{def:Sh}
    S(\lambda)=
    \begin{pmatrix}
      a(\lambda) & -\ol b(\lambda) \cr
    b(\lambda) & a(\lambda)\cr
    \end{pmatrix}
;\quad a(\lambda)=e^{-\pi\frac{\lvert\lambda\rvert^2}{ 2}}\quad ;\quad
a(\lambda)^2+|b(\lambda)|^2=1.
    \end{equation}
    \end{proposition}

\begin{proof}[Proof of Proposition~\ref{prop:ex}]

Let us  use the relation between the classical trajectories  to
identify $\nu^{+,in}$ (resp. $\nu^{-,in}$) as the two-scale Wigner
measure of $\alpha^\eps_1$ (resp.~$\alpha^\eps_2$) for
$\{z_1=0\}$, and $\nu^{+,out}$ (resp $\nu^{-,out}$) as the one
of~$\omega^\eps_2$ (resp.~$\omega_1^\eps$).

  \smallbreak

\noindent $\bullet$ If ($\alpha^\pm<\alpha^\mp=1/2$) or
($\alpha^\pm=1/2$ and $\eta_0^+\not=\eta_0^-$), the two incoming
measures are mutually singular. Therefore,  the measure of
$$\omega^\eps_1(z)=a\left(\frac{z_1}{\sqrt\eps}\right)\alpha^\eps_1(z) - \ol
b\left(\frac{z_1}{\sqrt\eps}\right) \alpha^\eps_2(z)$$ is the sum
of the measures of each term. We obtain
$$\nu^{-,out}(z,\zeta,\lambda)=a(\lambda)^2\nu^{+,in}
(z,\zeta,\lambda)+|b(\lambda)|^2\nu^{-,in}(z,\zeta,\lambda),$$
which gives in the variables  $(x,\xi,\eta)$
$$\nu^{-,out}(x,\xi,\eta)=T(\eta)\nu^{+,in}(x,\xi,\eta)+
\left(1-T(\eta)\right)\nu^{-,in}(x,\xi,\eta),$$ where we have used
$$a\left(\eta\,|\xi^*|^{-3/2}\right)^2=T(\eta),\;\;
\left|b\left(\eta\,|\xi^*|^{-3/2}\right)\right|^2=1-T(\eta).$$
Similarly, we get
$$\nu^{+,out}(x,\xi,\eta)=\left(1-T(\eta)\right)\nu^{+,in}(x,\xi,\eta)
+T(\eta)\nu^{-,in}(x,\xi,\eta).$$ We see here how the singularity
relation plays a role for finite $\eta$ in
Proposition~\ref{prop1}.

\noindent $\bullet$ If $\alpha^\pm<1/2$, both incident measures are
localized in
$\eta=\infty$, Prop.~\ref{prop1}
yields
$$\nu^{\pm,out}(x,\xi,\eta)=\nu^{\pm,in}(x,\xi,\eta).$$

\noindent $\bullet$  If $\alpha^\pm=1/2$ and $\eta_0^+=\eta_0^-$.
Observing that $\nu^{\pm,out}$ is localized above
$\eta_0^+=\eta_0^-:=\eta_0$ and taking into account the two-scale
joint measure $\theta$ between $\alpha^\eps_1$ and $\alpha^\eps_2$
which is also supported above $\eta_0$, we obtain
\begin{align*}
  \nu^{+,out}&=\left(1-T(\eta_0)\right)\nu^{+,in}
+T(\eta_0)\nu^{-,in}
+2\RE\left(a\left(\eta_0\,|\xi^*|^{-3/2}\right)b
\left(\eta_0\,|\xi^*|^{-3/2}\right)\theta\right),\\
\nu^{-,out}&=T(\eta_0)\nu^{+,in}+\left(1-T(\eta_0)\right)\nu^{-,in}
-2\RE\left(a\left(\eta_0\,|\xi^*|^{-3/2}\right)b
\left(\eta_0\,|\xi^*|^{3/2}\right)\theta\right).
\end{align*}
There exists~$(\rho_0,\phi_0)\in\R^+\times[0,2\pi]$ such that
$$a\left(\eta_0\,|\xi^*|^{-3/2}\right)b
\left(\eta_0\,|\xi^*|^{-3/2}\right)=\rho_0\,{e}^{i\phi_0}
.$$
One then observes that if one multiplies
the minus component of the initial data by ${e}^{i\phi}$, one
turns $\alpha^\eps_2$ into ${e}^{i\phi}\alpha^\eps_2$ and $\theta$
into ${e}^{-i\phi}\theta$, whence the result.
\end{proof}

\end{document}